\pdfoutput=1
\RequirePackage{ifpdf}
\ifpdf 
\documentclass[pdftex]{sigma}
\else
\documentclass{sigma}
\fi

\usepackage{tikz}
\usepackage{tikz-cd}
\usepackage{mathrsfs}
\usepackage{bm}
\usepackage{slashed}

\newcommand{\EE}{{\mathcal{E}}}
\newcommand{\II}{{\mathcal{I}}}
\newcommand{\LL}{{\mathcal{L}}}

\newcommand{\g}{{\mathfrak{g}}}

\newcommand{\su}{{\mathfrak{su}}}

\newcommand{\uu}{\mathfrak{u}}

\newcommand{\SU}{{\mathrm{SU}}}

\newcommand{\U}{{\mathrm{U}}}
\newcommand{\ii}{{\mathrm{i}}}

\newcommand{\R}{{\mathbb{R}}}

\newcommand{\Z}{{\mathbb{Z}}}

\newcommand{\C}{{\mathbb{C}}}

\newcommand{\1}{{\mathbb{1}}}

\newcommand{\beq}{\begin{equation}}
\newcommand{\eeq}{\end{equation}}
\newcommand{\bea}{\begin{eqnarray}}
\newcommand{\eea}{\end{eqnarray}}
\newcommand{\bal}{\begin{align}}
\newcommand{\eal}{\end{align}}
\newcommand{\bml}{\begin{multline}}
\newcommand{\eml}{\end{multline}}

\newcommand{\bdy}{\partial}

\newcommand{\lto}{\longrightarrow}

\def \d{\mathrm{d}}
\newcommand{\ol}{\overline}
\newcommand{\tr}{{\operatorname{tr}}}
\newcommand{\rk}{{\operatorname{rk}}}

\newcommand{\ip}[1]{\langle#1\rangle}

\makeatletter
\newcommand\xleftrightarrow[2][]{%
 \ext@arrow 9999{\longleftrightarrowfill@}{#1}{#2}}
\newcommand\longleftrightarrowfill@{%
 \arrowfill@\leftarrow\relbar\rightarrow}
\makeatother

\numberwithin{equation}{section}

\newtheorem{Theorem}{Theorem}[section]
\newtheorem*{Theorem*}{Theorem}

\newtheorem{Lemma}[Theorem]{Lemma}
\newtheorem{Proposition}[Theorem]{Proposition}
 { \theoremstyle{definition}

\newtheorem{Remark}[Theorem]{Remark} }

\begin{document}
\allowdisplaybreaks

\renewcommand{\thefootnote}{}

\newcommand{\arXivNumber}{2303.02623}

\renewcommand{\PaperNumber}{071}

\FirstPageHeading

\ShortArticleName{Geometry of Gauged Skyrmions}

\ArticleName{Geometry of Gauged Skyrmions\footnote{This paper is a~contribution to the Special Issue on Topological Solitons as Particles. The~full collection is available at \href{https://www.emis.de/journals/SIGMA/topological-solitons.html}{https://www.emis.de/journals/SIGMA/topological-solitons.html}}}

\Author{Josh CORK~$^{\rm a}$ and Derek HARLAND~$^{\rm b}$}

\AuthorNameForHeading{J.~Cork and D.~Harland}

\Address{$^{\rm a)}$~School of Computing and Mathematical Sciences,
University of Leicester,\\
\hphantom{$^{\rm a)}$}~University Road, Leicester, UK}
\EmailD{\href{mailto:josh.cork@leicester.ac.uk}{josh.cork@leicester.ac.uk}}
\URLaddressD{\url{https://le.ac.uk/people/josh-cork}}

\Address{$^{\rm b)}$~School of Mathematics, University of Leeds, Woodhouse Lane, Leeds, UK}
\EmailD{\href{mailto:d.g.harland@leeds.ac.uk}{d.g.harland@leeds.ac.uk}}

\ArticleDates{Received March 12, 2023, in final form September 14, 2023; Published online October 01, 2023}

\Abstract{A work of Manton showed how skymions may be viewed as maps between riemannian manifolds minimising an energy functional, with topologically non-trivial global minimisers given precisely by isometries. We consider a generalisation of this energy functional to gauged skyrmions, valid for a broad class of space and target 3-manifolds where the target is equipped with an isometric $G$-action. We show that the energy is bounded below by an equivariant version of the degree of a map, describe the associated BPS equations, and discuss and classify solutions in the cases where $G=\U(1)$ and $G=\SU(2)$.}


\Classification{53C07; 70S15; 53C43}

\renewcommand{\thefootnote}{\arabic{footnote}}
\setcounter{footnote}{0}

\section{Introduction}
The Skyrme model of nuclear physics was introduced in \cite{skyrme1962nucl} as a nonlinear sigma model, in which the pion field is encoded by a map $\phi\colon \R^3\to S^3$ constrained to be constant at infinity. Within this framework, baryons are identified as topological solitons in the theory called skyrmions, and baryon number is represented by a topological invariant, namely the degree of the map $\phi$. Some time later this identification was given a firm foundation when the model was shown to be an~effective model of QCD in limit of a large number of colours \cite{witten1979baryons}. Despite its relative simplicity, over the years the model has been shown to successfully reproduce many realistic properties of nuclei \cite{manton2022skyrmions:theory}. An important feature of the Skyrme model is its topological energy bound \cite{faddeev1976some}, which states that the energy of a map $\phi$ is bounded from below by a multiple of its degree.\looseness=-1

Although originally described for maps $\phi\colon \R^3\to S^3$, in \cite{Manton1987geometry} Manton showed how to formulate the Skyrme model for maps between arbitrary riemannian $3$-manifolds. This geometrised Skyrme model still admits a topological energy bound. The maps $\phi\colon M\to N$ that attain the bound and minimise their energy are either constant maps or isometries \cite{Manton1987geometry}. So the geometrised Skyrme model succinctly explains why the energy of a non-constant map $\phi\colon \R^3\to S^3$ is always strictly greater than its topological energy bound.

The Skyrme field $\phi\colon \R^3\to S^3$ can be coupled to a gauge field \cite{ArthurTchrakian1996GaugedSky,
BrihayeHartmannTchhrakian2001MonopolesGaugeSky, CallanWitten1984monopole,cork2018skyrmions,
CorkHarlandWinyard2021gaugedskyrmelowbinding,
criadoKhozeSpannowsky2021emergence,dHokerFarhi1984decoupling,LivramentoRaduShnir2023solitons,navarroLRaduTchrakian2019topological,
PietteTchrakian2000static,RaduTchrakian2006spinning}. Much of the work on the gauged Skyrme model seeks to model interactions between nucleons and electromagnetic fields, in which case the gauge group $\U(1)$ is most natural, but other choices of gauge group have also been studied.

An important gap in the current understanding of gauged skyrmions is a general geometric framework akin to Manton's treatment of ordinary skyrmions. In particular, although topological bounds are known for gauged Skyrme energies, to the best of our knowledge, there is no systematic understanding of existence or non-existence of solutions of the corresponding BPS equations. The aim of this article is to address both of these gaps in the literature. As we shall see, the gauged geometrised Skyrme model is much richer than its ungauged counterpart, and there are many more solutions besides isometries and constant maps. Despite this complexity, we are able to classify solutions for important choices of gauge group: see Theorems \ref{thm:classification-S1} and \ref{thm:classification-S2}.

A brief outline of the remainder of this paper is as follows. In Section \ref{sec:geometry-gauged-sky}, we discuss equivariant differential forms and equivariant topological degree, which act as the necessary geometric ingredients for gauging the Skyrme energy and its topological energy bound. In Section \ref{sec:BPS-eqns}, we use the ideas reviewed in Section \ref{sec:geometry-gauged-sky} to write down an energy functional for gauged skyrmions, a~topological energy bound, and corresponding BPS equations. Sections \ref{sec:princ-orbit-S1}, \ref{sec:princ-orbit-S2}, and \ref{sec:princ-orbit-S3} are devoted to studying the gauged Skyrme model, and analysis of solutions of the BPS equations, in the cases of gauge structure group $G=\U(1)$ and $G=\SU(2)$. We have divided analysis by the structure of the principal orbit space of the $G$ action on $N$. There are three cases of principal orbit which we consider: $\U(1)\cong S^1$, $\SU(2)/\U(1) \cong S^2$, and $\SU(2)\cong S^3$, however it is only in the first two cases that we analyse solutions of the BPS equations, whereas in the third case we argue that our BPS formalism is not applicable. Although not the motivation of the present article, it should be remarked that, from a physical standpoint, these three cases correspond to the gauging of electromagnetism, isospin symmetry, and weak symmetry respectively.
\section{Equivariant differential forms}\label{sec:geometry-gauged-sky}
In this, and the next section, we propose a generalisation of Skyrme fields and the Skyrme energy to that of a gauge theory. The broad geometric setting is as follows. We consider two riemannian 3-manifolds $(M,g_M)$ and $(N,g_N)$. Now let $G$ be a Lie group which acts isometrically on $(N,g_N)$, and consider a principal $G$ bundle $P\to M$ over $M$. A gauged Skyrme field then consists of a pair $(\phi,A)$, where $\phi$ is a smooth section of the associated bundle $P\times_GN$, and $A$ is a connection on $P$. As $\phi$ is not simply a function from $M$ to $N$, we cannot define its degree in the usual way using (de Rham) cohomology. Instead, we will make use of the equivariant cohomology of $N$.

The importance of equivariant cohomology in gauge theory was first highlighted in the classic paper of Atiyah--Bott \cite{atiyahbott1983YMriem}. Its appearance in the topological solitons literature is somewhat uncommon; notable exceptions include papers on Yang--Mills--Higgs vortices \cite{CieliebakGaioMundetIRieraSalamon2002symplectic,mundet1999yang,SpeightRomao2020}, and on gauged sigma models describing magnetic skyrmions \cite{walton2020geometry}. All of the aforementioned examples focus on two-dimensional field theories, in contrast to the three-dimensional focus of the present work. Therefore, it is useful to review the topic. The purpose of this section is to describe Cartan's model of equivariant cohomology, and how it leads to a topological degree. This material is not new, but is included to establish conventions and keep the discussion self-contained. For a more general treatment, see, e.g., \cite{atiyahbott1984moment,berlineGetzlerVergne2003}.

\subsection{Equivariant cohomology}
Cartan's model of equivariant cohomology of $N$ is based on {equivariant differential forms} over~$N$. These are $G$-invariant elements of $S^\ast \mathfrak{g}^\ast\otimes \Omega^\ast N$. Here $S^\ast\mathfrak{g}^\ast$ is the algebra of symmetric tensors in~$\mathfrak{g}^\ast$, on which $G$ acts coadjointly. Equivalently, an equivariant differential form is a~$G$-equi\-vari\-ant polynomial function from $\mathfrak{g}$ to $\Omega^\ast N$. If $Q\in S^p\g^\ast$ and $\alpha\in\Omega^qN$ we define $\deg(Q\otimes \alpha)=2p+q$. This gives the algebra of equivariant differential forms a grading.

Similarly, given any vector bundle $E\to N$ equipped with an action of $G$, an $E$-valued equivariant differential form is a $G$-invariant element of $S^\ast\g^\ast\otimes\Omega^\ast(N,E)$. An important example of such a form is the linear homomorphism $\nu\colon \g\to \Gamma(TN)$ which generates the action of~$G$ on~$N$, i.e., $\nu(X)$ is the Killing field associated to the action of $G$. The map $\nu$ is an equivariant $2$-form valued in $TN$; it is equivariant because
\begin{align*}
 \LL_{\nu(X)}\nu(Y)=[\nu(X),\nu(Y)]=\nu([X,Y]),
\end{align*}
and is a $2$-form as it lies in $S^1\g^\ast\otimes\Gamma(TN)$.

The differential of an equivariant differential form $\beta\colon \g\to\Omega^\ast N$ is defined to be
\begin{align}\label{equivariant-exterior-derivative}
\d_\g\beta(X) = \d(\beta(X))-\iota_{\nu(X)}\beta(X)\qquad\text{for all}\ X\in \g,
\end{align}
in which $\d\colon \Omega^\ast N\to\Omega^\ast N$ is the usual differential, and $\iota_{\nu(X)}$ is contraction with respect to the Killing field $\nu(X)$. One may quickly verify that $\d_\g$ increases degree by one and squares to 0. The cohomology of $\d_\g$ is isomorphic to the topological equivariant cohomology $H^\ast_G(N,\R)$ of $N$, also known as the Borel cohomology, which is defined by the ordinary cohomology $H^\ast(EG\times_GN,\R)$, where $EG\to BG$ is the universal bundle over the classifying space $BG$.

\subsection{Equivariant pullback}
Now we describe how to pull back equivariant forms on $N$ using the section $\phi$ of $P\times_GN$ and connection $A$ on $P$. We choose a local trivialisation of our principal bundle $P$ and local coordinates~$x^\lambda$ on $M$ and $y^\mu$ on $N$. Then a section of $P\times_G N$ is represented by a function~$\phi^\mu(x)$, denoted by indices $\mu=1,\dots,\dim N$. We choose a basis $I_a$ for $\g$. Then a connection is represented by real 1-forms $A^a(x)$, with $a=1,\dots,\dim \g$. The curvature of this connection is represented by the 2-forms
\begin{equation*}
F^a=\d A^a+\tfrac{1}{2}f^{a}_{bc}A^b\wedge A^c,
\end{equation*}
where $f^a_{bc}$ are the structure constants of $\g$ defined by $[I_b,I_c]=f^a_{bc}I_a$. The curvature obeys the Bianchi identity
\begin{equation}\label{bianchi}
\d F^a + f^a_{bc}A^b\wedge F^c = 0.
\end{equation}
The left action of $G$ on $N$ induces the Lie algebra homomorphism
\begin{equation}\label{nu-definition}
\nu\colon \ \g\lto\Gamma(TN),\qquad X \mapsto \frac{\d}{\d t}\Big|_{t=0}\exp(-tX)\cdot y.
\end{equation}
By small abuse of notation, the vector fields associated with $I_a\in\g$ will be denoted $\nu(I_a)=I_a=I_a^\mu(y)\frac{\partial}{\partial y^\mu}\in\Gamma(TN)$. Then
\begin{equation}\label{Lie bracket}
I_a^\mu\partial_\mu I_b^\lambda - I_b^\mu \partial_\mu I_a^\lambda = f_{ab}^cI_c^\lambda,
\end{equation}
since $\nu$ is a homomorphism.

For an equivariant differential form $\beta\colon \g\to\Omega^q N$, we define coefficients $\beta_{a_1\dots a_p;\mu_1\dots\mu_q}$ such that
\begin{equation*}
\beta(X) = \frac{1}{p!q!}X^{a_1}\cdots X^{a_p}\beta_{a_1\dots a_p;\mu_1\dots\mu_q}\d y^{\mu_1}\wedge\dots\wedge\d y^{\mu_q},\qquad X=X^aI_a\in\g.
\end{equation*}
The coefficients $\beta_{a_1\dots a_p;\mu_1\dots\mu_q}$ are chosen totally symmetric in $a_i$ and totally antisymmetric in~$\mu_j$. The condition that $\beta$ is equivariant translates into
\begin{gather}
 0=I_b^\nu\partial_\nu\beta_{a_1a_2\dots a_p;\mu_1\mu_2\dots\mu_q}
+ \beta_{a_1a_2\dots a_p;\nu\mu_2\dots\mu_q}\partial_{\mu_1}I_b^\nu + \beta_{a_1a_2\dots a_p;\mu_1\nu\dots\mu_q}\partial_{\mu_2}I_b^\nu + \cdots \nonumber\\
 \hphantom{0=}{} +f^c_{a_1 b}\beta_{ca_2\dots a_p;\mu_1\mu_2\dots\mu_q}+f^c_{a_2 b}\beta_{a_1c\dots a_p;\mu_1\mu_2\dots\mu_q}+\cdots .\label{Lie derivative}
\end{gather}
We define the covariant differential of $\phi$ to be
\begin{equation}\label{cov-derivative}
\d^A\phi = \d^A\phi^\mu \frac{\partial}{\partial y^\mu} = (\d\phi^\mu(x) -A^a(x)I^\mu_a(\phi(x)))\frac{\partial}{\partial y^\mu}.
\end{equation}
This is a section of $T^\ast M\otimes\phi^\ast TN$ which can be understood more globally as $\phi^{\ast A}(I)$, where $I$ is the section of $T^\ast N\otimes TN$ associated with the identity map on $TN$. The pullback of an~equivariant differential form $\beta$ is then defined by
\begin{equation}\label{local pullback}
\phi^{\ast A}\beta = \frac{1}{p!q!}\beta_{a_1\dots a_p;\mu_1\dots\mu_q}(\phi(x)) F^{a_1}\wedge\dots\wedge F^{a_p}\wedge \d^A\phi^{\mu_1}\wedge\dots\wedge \d^A\phi^{\mu_q}.
\end{equation}
This is a differential form on $M$. Its degree is the same as that of $\beta$, namely $2p+q$, so equivariant pullback preserves grading. We will see in Proposition \ref{Prop:pullback-properties} below that it also intertwines the differentials $\d$ and $\d_\g$.

It is important to check that our definition \eqref{local pullback} of pullback is gauge-invariant. Infinitesimal local gauge transformations are given by $\g$-valued functions $\lambda=\lambda^a(x)I_a$, $a=1,\dots,\dim\g$. They induce gauge transformations on $(\phi,A)$ as $\phi\mapsto\exp(-t\lambda)\cdot \phi$, $A\mapsto\exp(-t\lambda)\cdot A := \exp(-t\lambda)(\d\exp(t\lambda)+A\exp(t\lambda))$. The corresponding infinitesimal actions are
\begin{align}
\dot{\phi}^\mu(x)&:=\frac{\d}{\d t}\Big|_{t=0}\exp(-t\lambda)\cdot \phi = \lambda^a(x)I_a^\mu(\phi(x)), \label{gt phi}\\
\dot{A}^a(x) &:=\frac{\d}{\d t}\Big|_{t=0}\exp(-t\lambda)\cdot A^a= \d\lambda^a(x) + f^a_{bc}A^b(x)\lambda^c(x)\label{gt A}.
\end{align}
We then have the following statement.

\begin{Proposition}\label{Prop:pullback-properties}
The equivariant pullback given in equation \eqref{local pullback} is gauge-invariant, and satisfies $\phi^{\ast A}\d_\g\beta = \d\phi^{\ast A}\beta$.
\end{Proposition}
\begin{proof}
For simplicity of presentation, we consider only the case $p=q=1$. Then
\begin{equation*}
\phi^{\ast A}\beta = \beta_{a;\mu}(\phi(x))F^a(x)\wedge \d^A\phi^\mu(x) .
\end{equation*}
The action of a gauge transformation on $\d^A\phi$ is calculated from \eqref{gt phi} and \eqref{gt A} using \eqref{Lie bracket} as follows:
\begin{align*}
\dot{\d^A\phi^\mu}&=\d \dot{\phi^\mu}-\dot{A}^aI_a^\mu-A^a\dot{\phi}^\nu\partial_\nu I_a^\mu \\
&=\d(\lambda^a I_a^\mu)-(\d\lambda^a+f^a_{bc}A^b\lambda^c)I_a^\mu - A^a\lambda^bI_b^\nu\partial_\nu I_a^\mu \\
&=\d\lambda^a I_a^\mu + \lambda^a\d\phi^\nu \partial_\nu I_a^\mu-\d\lambda^aI_a^\mu-A^b\lambda^c(f^a_{bc}I_a^\mu + I_c^\nu\partial_\nu I_b^\mu) \\
&=\lambda^a\d\phi^\nu \partial_\nu I_a^\mu-A^b\lambda^cI_b^\nu\partial_\nu I_c^\mu \\
&=\lambda^a \d^A\phi^\nu\partial_\nu I_a^\mu.
\end{align*}
Also the action of a gauge transformation on $F$ calculated from \eqref{gt A} is
\begin{equation*}
\dot{F^a}=f^a_{bc}F^b\lambda^c.
\end{equation*}
Therefore, the action on $\phi^{\ast A}\beta$ is{\samepage
\begin{align*}
\dot{\phi^{\ast A}\beta}&=\dot\phi^\nu\partial_\nu\beta_{a;\mu}F^a\wedge \d^A\phi^\mu + \beta_{a;\mu}\dot{F^a}\wedge \d^A\phi^\mu + \beta_{a;\mu}F^a\wedge \dot{\d^A\phi^\mu} \\
&=\lambda^bI_b^\nu\partial_\nu\beta_{a;\mu}F^a\wedge \d^A\phi^\mu + \beta_{a;\mu}f^a_{bc}F^b\lambda^c\wedge \d^A\phi^\mu + \beta_{a;\mu}F^a\wedge \lambda^b \d^A\phi^\nu\partial_\nu I_b^\mu \\
&=\lambda^b(I_b^\nu\partial_\nu\beta_{a;\mu} + f^c_{ab}\beta_{c;\mu} + \partial_\mu I_b^\nu\beta_{a;\nu})\,F^a\wedge \d^A\phi^\mu
=0
\end{align*}
by \eqref{Lie derivative}. So $\phi^{\ast A}\beta$ is gauge-invariant.}

Next we calculate exterior differentials. By definition \eqref{equivariant-exterior-derivative},
\begin{equation*}
\phi^{\ast A}\d_\g\beta = \partial_\nu\beta_{a;\mu}F^a\wedge \d^A\phi^\nu\wedge \d^A\phi^\mu - \beta_{a;\mu}I_b^\mu F^a\wedge F^b.
\end{equation*}
Therefore,
\begin{align*}
\d\phi^{\ast A}\beta &=\d\phi^\nu\partial_\nu\beta_{a;\mu}\wedge F^a\wedge \d^A\phi^\mu + \beta_{a;\mu}\d F^a\wedge \d^A\phi^\mu \nonumber\\
&\quad -\beta_{a;\mu}F^a\wedge\d A^b I_b^\mu + \beta_{a;\mu}F^a\wedge A^b\wedge \d\phi^\nu\partial_\nu I_b^\mu \\
&= (\d^A\phi^\nu+A^bI_b^\nu)\partial_\nu\beta_{a;\mu}\wedge F^a\wedge \d^A\phi^\mu
- \beta_{a;\mu}f^a_{bc}A^b\wedge F^c\wedge \d^A\phi^\mu \nonumber\\
&\quad -\beta_{a;\mu}F^a\wedge (F^b-\tfrac12 f^b_{cd}A^c\wedge A^d) I_b^\mu
+ \beta_{a;\mu}F^a\wedge A^b\wedge (\d^A\phi^\nu+A^cI_c^\nu) \partial_\nu I_b^\mu \\
&= \phi^{\ast A}\d_\g \beta + \beta_{a;\mu}F^a\wedge A^b\wedge A^c (\tfrac12 f_{bc}^dI_d^\mu +I_c^\nu\partial_\nu I_b^\mu) \nonumber \\
&\quad+ F^a\wedge A^b\wedge \d^A\phi^\mu (I_b^\nu\partial_\nu\beta_{a;\mu}+f_{ab}^c\beta_{c;\mu} +\beta_{a;\nu}\partial_\mu I_b^\nu )\\
&= \phi^{\ast A}\d_\g \beta,
\end{align*}
by \eqref{bianchi}, \eqref{Lie bracket} and \eqref{Lie derivative}.
\end{proof}

Although our definition of equivariant pullback was entirely local, the notion can also be described from an entirely global perspective. Very briefly, a connection $A$ on $P$ determines a~map $\Psi^A$ from the algebra of equivariant differential forms on $N$ to the algebra of differential forms on $P\times_GN$ \cite{berlineGetzlerVergne2003}. This map is known as the Chern--Weil homomorphism (in the special case where $N$ is a point and $G=\SU(n)$ acts trivially, it sends the polynomial $\tr(X^n)$ to the Chern--Weil form $\tr(F\wedge\dots\wedge F)$). In general, the equivariant pullback of an equivariant differential form $\beta$ on $N$ by a section $\phi\colon M\to P\times_N G$ is then given by $\phi^{\ast A}\beta = \phi^\ast \Psi^A(\beta)$. For more details on this, see \cite{berlineGetzlerVergne2003}.

\subsection{Equivariant topological degree}

We are now in a position to define an equivariant degree of a section $\phi$. For simplicity, we now assume that $M$ and $N$ are compact and three dimensional, and that $H^3_G(N,\R)$ is one dimensional and generated by
\begin{align*}
 X\mapsto V_N+\mu(X).
\end{align*}
Here $V_N$ is a $G$-invariant 3-form on $N$ and $\mu$ is an equivariant linear map $\g\to\Omega^1N$. This form must be $\d_\g$-closed, and hence must satisfy
\begin{align}\label{moment-map-def}
 & \d\mu(X)=\iota_{\nu(X)}V_N\qquad\text{and}\\
\label{moment-map-constraint}
 & \iota_{\nu(X)}\mu(X)=0\qquad\text{for all }X\in\g.
\end{align}
The first condition \eqref{moment-map-def} resembles the definition of the (co)-moment map found in symplectic geometry, where here the role of the symplectic form is replaced by the volume form. For this reason, we shall refer to such a map $\mu$ as a moment map associated to the action $\nu\colon \g\to\Gamma(TN)$.
The second condition \eqref{moment-map-constraint} is a significant constraint on the moment map that has no analogue in symplectic geometry; equivariance ensures $\iota_{\nu(X)}\mu(X)$ is constant, but not that this constant is zero. As we shall explore later, even for simple natural examples there are no non-trivial moment maps satisfying \eqref{moment-map-constraint}, ruling out the possibility of a suitably well-defined topological degree.\looseness=-1

Equations \eqref{moment-map-def} and \eqref{moment-map-constraint} have been considered before in the more general context of homotopy moment maps \cite{calliesfregierrogerszambon2016homotopy}, with no constraints on the dimension of $N$ or the degree of the form. Existence and uniqueness results are established in \cite{calliesfregierrogerszambon2016homotopy}, of which it is worth noting that in three dimensions, for compact and semisimple $G$, a sufficient condition for existence is when the $G$-action has a fixed point.

The $G$-invariant 3-form $V_N$ defines a non-trivial class in $H^3_G(N,\R)$ provided that $\int_N V_N\neq0$ (in addition to the constraints \eqref{moment-map-def} and \eqref{moment-map-constraint}). One could, for example, choose $V_N$ to be the volume form associated with the $G$-invariant metric $g_N$. Assuming that these conditions are satisfied, we define the {equivariant topological degree} of a section $\phi$ of $P\times_GN$ by
\begin{align}\label{equivariant-top-degree}
 \deg_G(\phi)=\frac{\int_M\phi^{\ast A}(V_N+\mu)}{\int_N V_N}.
\end{align}
As with the usual degree, we may remove the assumption of compactness of $M$ by imposing suitable boundary conditions on $P\to M$, $A$, and $\phi$.

In suitable circumstances, the topological degree $\deg_G(\phi)$ is an integer. We briefly explain why, using equivariant homology. The equivariant homology groups $H^G_\ast(N,\Z)$ of $N$ are defined to be the ordinary homology of $EG\times_G N$, in which $EG\to BG$ is the universal $G$-bundle over the classifying space $BG$. There is a natural map $N\to N\times_G EG$, and hence $H_\ast(N,\Z)\to H^G_\ast(N,\Z)$. We denote by $[N]$ the image in $H^G_3(N,\Z)$ of the fundamental class of $N$, and assume that this generates $H^G_3(N,\Z)$. The dual map in cohomology is $H^\ast_G(N,\R)\to H^\ast(N,\R)$ given by $[\beta(X)]\mapsto[\beta(0)]$ for an equivariant closed form $\beta$. By assumption, this map is an~isomorphism, so if $V_N+\mu(X)$ and $\tilde{V}_N+\tilde{\mu}(X)$ are two equivariantly closed differential forms that satisfy $\int_NV_N=\int_N\tilde{V}_N$, then their difference must be in the image of $\d_\g$. This means (by Proposition~\ref{Prop:pullback-properties}) that~\eqref{equivariant-top-degree} is independent of the choice of $V_N$ and $\mu$.

By definition, the bundle $P\to M$ can be obtained as the pullback of $EG\to BG$ by a map $j\colon M\to BG$. There are hence natural maps $P\to EG$ and $P\times_G N \to EG\times_G N$. So the section $\phi\to P\times_G N$ determines a map $M\to EG\times_G N$, and hence a map $\phi_{\ast A}\colon H_3(M,\Z)\to H_3^G(N,\Z)$ that is dual to $\phi^{\ast A}\colon H^3_G(N,\R)\to H^3(M,\R)$. The image under this map of the fundamental class $[M]\in H_3(M,\Z)$ is by assumption equal to $n$ times $[N]$ for some integer $n$. Then $\int_M\phi^{\ast A}(V_N+\mu)$ is equal to the pairing of $n[N]$ with $[V_N+\mu]$, which is in turn equal to $n\int_N V_N$.

\section{BPS equations for gauged skyrmions}\label{sec:BPS-eqns}
We shall now use the geometric ingredients reviewed in the previous section to describe a~gauged version of the Skyrme model, valid for a broad class of riemannian 3-manifolds $(M,g_M)$ and $(N,g_N)$, and derive its BPS equations. To motivate our model, we begin by reviewing Manton's geometrised Skyrme model \cite{Manton1987geometry} from a novel perspective.
\subsection{The ungauged Skyrme model}
Let $(M,g_M)$ and $(N,g_N)$ be two riemannian 3-manifolds. The Hodge star of $N$ is a map $\star_{N}\colon \Lambda^1N\to\Lambda^2N$. This corresponds to the section $\Sigma$ of $TN\otimes\Lambda^2N$ defined by
\begin{align}\label{Hodge-star-def}
 g_N(u,\Sigma(v,w))=V_N(u,v,w) \qquad\text{for all }\,u,v,w\in TN,
\end{align}
with $V_N$ the volume form on $(N,g_N)$. The Skyrme energy of a map $\phi\colon M\to N$ is defined by
\begin{align}\label{Skyrme-energy}
E[\phi]=\int_Mg_N(\d\phi\wedge\star_{M} \d\phi) + g_N(\phi^\ast\Sigma\wedge\star_{M}\phi^\ast\Sigma).
\end{align}
This is equivalent to the energy introduced by Manton \cite{Manton1987geometry}. In this context, the map $\phi$ is referred to as a Skyrme field. Note that $\d\phi$ and $\phi^\ast\Sigma$ are sections of $\Lambda^1 M\otimes \phi^\ast TN$ and $\Lambda^2 M\otimes \phi^\ast TN$ respectively, and our expression for the energy uses the metric $g_N$ on $TN$. The topological degree of $\phi$ is the quantity
\begin{align}\label{top-degree}
 \deg(\phi)=\frac{1}{{\rm Vol}(N)}\int_M\phi^\ast V_N.
\end{align}
A skyrmion is a critical point of the Skyrme energy \eqref{Skyrme-energy}, and these are classified by their topological degree. The energy \eqref{Skyrme-energy} is bounded below proportionally by the degree \eqref{top-degree}. To derive this energy bound, consider the identity
\begin{gather}
g_N ((\d\phi-\star_{M}\phi^\ast\Sigma)\wedge(\star_{M} \d\phi-\phi^\ast\Sigma) )\nonumber\\
\qquad{} =g_N(\d\phi\wedge\star_{M} \d\phi) + g_N(\phi^\ast\Sigma\wedge\star_{M}\phi^\ast\Sigma) - 2g_N(\d\phi\wedge\phi^\ast\Sigma).\label{BOG-argument-skyrme}
\end{gather}
Employing an orthonormal frame $E_a$ for $TN$, and a dual frame $e^a$ for $\Lambda^1N$, we can write the last term in \eqref{BOG-argument-skyrme} as
\begin{align}\label{ip-volume-pullback}
-2g_N(E_a,E_b)\phi^\ast e^a \wedge\phi^\ast(\star_{N} e^b)=-2\phi^\ast (e^a\wedge\star_{N} e^a)=-6\phi^\ast V_N.
\end{align}
Thus, as \eqref{BOG-argument-skyrme} is non-negative, this leads to the topological energy bound
\begin{align}
E[\phi]\geq 6{\rm Vol}(N)|\deg\phi|,\label{Top-bound-skyrmions}
\end{align}
which is saturated if and only if
\begin{align}\label{ungauged bog eq}
\star_{M} \d\phi = \phi^\ast \Sigma.
\end{align}
Solutions of this equation, if they exist, represent global minimisers of the Skyrme energy \eqref{Skyrme-energy} within the space of maps of fixed degree. Equation \eqref{ungauged bog eq} is equivalent to the statement that the pullback $\phi^\ast$ commutes with the Hodge star when acting on 1-forms, i.e., that the diagram
\begin{equation}\label{BPS-comm-diagram}
\begin{tikzcd}
\Lambda^1N\arrow{r}{\phi^\ast}\arrow[swap]{d}{\star_{N}}&\Lambda^1M\arrow{d}{\star_{M}}\\
\Lambda^2N\arrow{r}{\phi^\ast}&\Lambda^2M
\end{tikzcd}
\end{equation}
commutes. Since the maps $\star$ are isomorphisms, this holds only if the ranks of the maps $\phi^\ast\colon \Lambda^1N\to\Lambda^1M$ and $\phi^\ast\colon \Lambda^2N\to\Lambda^2M$ are equal. If $\phi^\ast\colon \Lambda^1N\to\Lambda^1N$ has rank~2 (resp.~1) then $\phi^\ast\colon \Lambda^2N\to\Lambda^2N$ has rank 1 (resp.\ 0), because if $u\in\Lambda^1N$ is such that $\phi^\ast u=0$, then $\phi^\ast(u\wedge v)=0$ for all $v\in\Lambda^1N$.
So this leaves only two cases: either the maps are both zero rank, or full rank. The former case is when the map $\phi$ is constant. The latter case is equivalent to $\phi$ being a diffeomorphism, and in this case \eqref{BPS-comm-diagram} commutes precisely when $\phi$ is an isometry.

This result is a reformulation of the result of Manton \cite{Manton1987geometry}. The equation \eqref{ungauged bog eq} is equivalent to Manton's equations $\lambda_1=\lambda_2\lambda_3$, $\lambda_2=\lambda_3\lambda_1$, $\lambda_3=\lambda_1\lambda_2$ relating the eigenvalues $\lambda_i^2$ of the strain tensor $g_M^{-1}\circ\phi^\ast g_N$. The only solutions of these equations are $\lambda_i=0$ corresponding to $\phi$ constant, or $\lambda_i=1$ corresponding to $\phi$ an isometry. An important consequence of this result is that non-constant ordinary skyrmions $\phi\colon \R^3\to S^3$ cannot attain the topological bound \eqref{Top-bound-skyrmions} as there is no isometry between $\R^3$ and $S^3$.
\subsection{Gauged Skyrme energy functional}\label{sec:energy-gauged}
Now we assume that $(N,g_N)$ admits an isometric action generated by $\nu\colon \g\to \Gamma(TN)$ and let $P\to M$ be a $G$-principal bundle. Consider the following energy functional for a section $\phi$ of~$P\times_GN$ and a connection $A$ on $P$:
\begin{align} E[\phi,A]=&\int_Mc_1\bigl|\d^A\phi\bigr|^2+c_2\bigl|\phi^{\ast A}\Sigma\bigr|^2+c_3\bigl|\phi^{\ast A}\nu\bigr|^2+c_4|\phi^{\ast A}\mu^\sharp|^2\nonumber\\
&+\bigl\langle c_5\phi^{\ast A}\nu+c_6\phi^{\ast A}\mu^\sharp,\phi^{\ast A}\Sigma\bigr\rangle.\label{gauged-Skyrme-energy}
\end{align}
Here $c_1,\dots,c_6\in\R$ are constants, $\Sigma\in\Omega^2(N,TN)$ is the $2$-form (defined in \eqref{Hodge-star-def}) representing the Hodge star on $N$ as was the case for the ordinary Skyrme energy \eqref{Skyrme-energy}, and $\mu^\sharp\colon \g\to\Gamma(TN)$ is the vector field dual to the moment map $\mu$ (defined by the conditions \eqref{moment-map-def} and \eqref{moment-map-constraint}) via the metric $g_N$, i.e.,
\begin{align*}
 g_N(\mu^\sharp(X),v)=\mu(X)(v)\qquad\text{for all }X\in\g,\ v\in TN.
\end{align*}
Also, the inner product $\ip{\cdot,\cdot}$ used in \eqref{gauged-Skyrme-energy} is defined as usual by the metrics $g_M$ and $g_N$ as
\begin{align}
 \ip{\alpha,\beta}=g_N(\alpha\wedge\star_{M}\beta),\qquad|\alpha|^2=\ip{\alpha,\alpha}.\label{pairing}
\end{align}
Although clearly a lot more complex than the ordinary Skyrme energy \eqref{Skyrme-energy}, the energy functional \eqref{gauged-Skyrme-energy} possesses some important natural features, which are similar in spirit to \eqref{Skyrme-energy}. Firstly it respects the obvious symmetries; it is gauge-invariant by Proposition \ref{Prop:pullback-properties}, and invariant under orientation-preserving isometries of $g_M$ and $g_N$ by construction. It also shares the feature of the ordinary Skyrme energy \eqref{Skyrme-energy} that all terms are either quadratic or quartic in derivatives; the quadratic term is $|\d^A\phi|^2$, and all others are quartic in derivatives since they are formed by pulling back the equivariant $2$-forms $\Sigma$, $\nu$, and $\mu^\sharp$, and then pairing them using \eqref{pairing}.

The attentive reader may wonder why we have not considered a term of the form $\ip{\phi^{\ast A}\nu,\phi^{\ast A}\mu^\sharp}$ in \eqref{gauged-Skyrme-energy}; indeed, if the motivation is to consider all quartic terms made from combinations of the equivariant $2$-forms $\Sigma$, $\nu$, and $\mu^\sharp$, this should be there. However, recall that part of the definition of $\mu$ is the equation \eqref{moment-map-constraint} required so that $V_N+\mu$ is equivariantly closed. It is straightforward to show that this condition, $\iota_{\nu(X)}\mu(X)=0$ for all $X\in\g$, implies $\ip{\phi^{\ast A}\nu,\phi^{\ast A}\mu^\sharp}=0$ identically. Such a term might be relevant for theories which do not possess topological degree, but since the motivation of the present article is to describe BPS gauged skyrmions, we include the condition~\eqref{moment-map-constraint}, and hence do not have such a term in our energy.

One natural term which is not apparently present in the functional \eqref{gauged-Skyrme-energy} as presented is a~Yang--Mills term, i.e., a term proportional to $|F|_\g^2$, with $F$ being the curvature of $A$. A Yang--Mills term in \eqref{gauged-Skyrme-energy} can be reproduced by suitable combination of $|\phi^{\ast A}\nu|^2$ and $|\phi^{\ast A}\mu^\sharp|^2$, both of which are quadratic in $F$. A caveat to this is that the Yang--Mills term may be accompanied by additional non-trivial couplings between $F$ and $\phi$. This we shall see explicitly via an example below.

We shall now illustrate how the functional \eqref{gauged-Skyrme-energy} generalises previously studied energy functionals for gauged skyrmions. Specifically, here we consider the case where $N=\SU(2)$ equipped with the round metric on $S^3\cong\SU(2)$, and $G=\SU(2)$ acting on itself via the adjoint action. As we shall see, in this case the energy \eqref{gauged-Skyrme-energy} may be reduced to the form of the energy introduced in \cite{CorkHarlandWinyard2021gaugedskyrmelowbinding}. See also \cite[Section~8.2]{calliesfregierrogerszambon2016homotopy}.

For all $g\in\SU(2)$ and $X\in\su(2)$, the left and right actions of $\SU(2)$ on itself induce two Lie algebra homomorphisms $\nu_{L,R}\colon \su(2)\to\Gamma(T\SU(2))$,
\begin{equation*}
\nu_L(X)=\frac{\d}{\d t}\Big|_{t=0}\exp(-tX)g,\qquad\nu_R(X)=\frac{\d}{\d t}\Big|_{t=0}g\exp(tX).
\end{equation*}
The adjoint action is described by
\begin{equation*}
\nu=\nu_L+\nu_R.
\end{equation*}
The left- and right-invariant Maurer--Cartan 1-forms $\theta_{L,R}\in\su(2)\otimes\Gamma(\Lambda^1\SU(2))$ are defined by $\theta_L=g^{-1}\d g$, $\theta_R=(\d g)g^{-1}$ and satisfy the structure equations
\begin{equation}\label{structure-eqns-SU2}
\d\theta_L+\theta_L\wedge\theta_L=0,\qquad\d\theta_R-\theta_R\wedge\theta_R=0.
\end{equation}
Their contractions with $\nu_{L,R}$ satisfy{\samepage
\begin{equation}\label{contraction-Maurer}
\iota_{\nu_R(X)}\theta_L=-\iota_{\nu_L(X)}\theta_R=X,\qquad \iota_{\nu_L(X)}\theta_L = -g^{-1}Xg,\qquad \iota_{\nu_R(X)}\theta_R = gXg^{-1}
\end{equation}
for all $X\in\su(2)$, $g\in\SU(2)$.}

The round metric on $\SU(2)$ is bi-invariant and takes the form
\begin{equation*}
g_N=-\tfrac{1}{2}\tr(\theta_L\theta_L)=-\tfrac{1}{2}\tr(\theta_R\theta_R).
\end{equation*}
The associated volume form is
\begin{equation*}
V_N=-\tfrac{1}{12}\tr(\theta_L\wedge\theta_L\wedge\theta_L)=-\tfrac{1}{12}\tr(\theta_R\wedge\theta_R\wedge\theta_R).
\end{equation*}
Then from \eqref{contraction-Maurer} and the structure equations \eqref{structure-eqns-SU2}
\begin{equation*}
\iota_{\nu(X)}V_N = \tfrac{1}{4}\tr(X(\theta_R\wedge\theta_R-\theta_L\wedge\theta_L))
=\d\tfrac{1}{4}\tr(X(\theta_R+\theta_L)).
\end{equation*}
So a moment map is given by $\mu(X)=\tfrac{1}{4}\tr(X(\theta_L+\theta_R))$. By \eqref{contraction-Maurer}, this satisfies \eqref{moment-map-constraint}, and so allows for a suitably well-defined topological degree. In fact, as we shall explore later, it is not the most general choice of moment map for the adjoint action, however this choice suffices for the purpose of this illustration.

The dual $\mu^\sharp$ is given by
\begin{equation}\label{moment-map-adjoint-round-sphere-choice}
\mu^\sharp(X)=\tfrac{1}{2}(\nu_L(X)-\nu_R(X)),
\end{equation}
because
\begin{align*}
 g_N(\cdot,\nu_L(X)-\nu_R(X))& =-\tfrac{1}{2}\tr(\theta_L(\iota_{\nu_L(X)}-\iota_{\nu_R(X)})\theta_L)\\
 & =\tfrac{1}{2}\tr(g\theta_Lg^{-1}X+X\theta_L)=\tfrac{1}{2}\tr(X(\theta_R+\theta_L)).
\end{align*}
The forms $\Sigma$, $\nu$, $\mu^\sharp$ are $T\SU(2)$-valued. In order to write down their pullbacks, it is convenient to identify $T\SU(2)$ with the trivial bundle $\su(2)\times \SU(2)$. Any $T\SU(2)$-valued form can be turned into an $\su(2)$-valued form by composing it with the Maurer--Cartan form $\theta_L\colon T\SU(2)\to\su(2)$. So $\nu(X)$ and $\mu^\sharp(X)$ become the $\su(2)$-valued functions
\begin{equation*}
\iota_{\nu(X)}\theta_L = X-g^{-1}Xg,\qquad \iota_{\mu^\sharp(X)}\theta_L = -\tfrac{1}{2}(X+g^{-1}Xg)
\end{equation*}
of $g\in\SU(2)$. By the same token, $\Sigma$ becomes
\begin{equation*}
\iota_\Sigma \theta_L=\tfrac{1}{2}\theta_L\wedge\theta_L,
\end{equation*}
because
\begin{equation*}
-\tfrac{1}{2}\tr(\iota_{\Sigma(u,v)}\theta_L\,X) = g_N(\Sigma(u,v),\nu_R(X)) = V_N(u,v,\nu_R(X)) = -\tfrac{1}{4}\tr([\theta_L(u),\theta_L(v)]X)
\end{equation*}
for all $u$, $v\in T_g\SU(2)$ and $X\in\su(2)$.

Finally, the tautological section $I$ of $T\SU(2)\otimes\Lambda^1\SU(2)$ corresponds to the $\su(2)$-valued \mbox{1-form} $\theta_L$. It then follows that, given a map $U\colon M\to \SU(2)$ (which represents $\phi$) and an $\su(2)$ connection $A$, the equivariant pullbacks are given by the $\su(2)$-valued forms
\begin{gather*}
\d^A\phi=\phi^{\ast A}(I)=L^A:=U^{-1}(\d U+[A,U]),\\
\phi^{\ast A}(\Sigma)=\tfrac{1}{2}\bigl(L^A\wedge L^A\bigr),\\
\phi^{\ast A}(\nu)=F-U^{-1}F U,\\
\phi^{\ast A}(\mu^\sharp) = -\tfrac{1}{2}\bigl(F+U^{-1}F U\bigr).
\end{gather*}
Putting this together, the energy functional \eqref{gauged-Skyrme-energy} becomes
\begin{align}\label{CHW-energu}
E[U,A]=&\int_Mc_1\big|L^A\big|^2+\tfrac{c_2}{4}\big|L^A\wedge L^A\big|^2+\tfrac{1}{2}(4c_3+c_4)|F|^2+\tfrac{1}{2}(c_4-4c_3)\big\langle F,U^{-1}FU\big\rangle\nonumber\\
&+\tfrac{1}{4}\big\langle (2c_5-c_6)F-(2c_5+c_6)U^{-1}F U,L^A\wedge L^A\big\rangle,
\end{align}
which is equivalent to the energy introduced in \cite{CorkHarlandWinyard2021gaugedskyrmelowbinding}. If $c_4=4c_3$, then the quadratic term in $F$ is just the usual Yang--Mills term.

\subsection{Energy bound and BPS equations}

In this section, we discuss how to bound the energy \eqref{gauged-Skyrme-energy} by the equivariant topological degree~\eqref{equivariant-top-degree}. To do this it is useful to write the integrand $\phi^{\ast A}(V_N+\mu)$ of \eqref{equivariant-top-degree} using the pairing~\eqref{pairing}. By definition (equations \eqref{Hodge-star-def} and \eqref{moment-map-def}), we may show analogously to the argument in \eqref{ip-volume-pullback} that
\begin{align*}
 V_N+\mu=\tfrac{1}{3}g_N\big(I\wedge\big(\Sigma+3\mu^\sharp\big)\big)
\end{align*}
(with $I\in\Gamma(TM\otimes T^\ast M)$ being the identity map). Pulling this back, we therefore find
\begin{align}\label{charge-density-ip}
 \phi^{\ast A}(V_N+\mu)=\tfrac{1}{3}\bigl\langle \star_{M}\d^A\phi,\phi^{\ast A}\big(\Sigma+3\mu^\sharp\big)\bigr\rangle.
\end{align}
In \cite{CorkHarlandWinyard2021gaugedskyrmelowbinding}, general energy bounds for the $\SU(2)$ gauged Skyrme energy \eqref{CHW-energu} were derived. The arguments presented there did not explicitly rely on the $\SU(2)$ structure~-- they only needed that the energy and charge may be expressed as quadratic forms, just as we have done here~-- and so may be applied immediately to our energy. Using the methods outlined in \cite{CorkHarlandWinyard2021gaugedskyrmelowbinding}, one may use~\eqref{charge-density-ip} to show that (assuming suitable constraints on the parameters $c_i$ so that $E$ is positive) one has the general bound\footnote{This equates to the bound given in \cite{CorkHarlandWinyard2021gaugedskyrmelowbinding} for the case $N=\SU(2)$ under the identification
\[
c_1=x_1,\qquad c_2 =x_2,\qquad c_3=\tfrac{x_5}{4},\qquad
 c_4 =\tfrac{x_6}{2},\qquad c_5=-\tfrac{x_4}{2},\qquad c_6 =\tfrac{x_3}{2},
 \]
up to an overall constant factor in the energy expression.}
\begin{align*}
E[\phi,A]\geq6\sqrt{\frac{c_1\big(c_3\big(4c_2c_4-c_6^2\big)-c_4c_5^2\big)}{4c_3(9c_2+c_4-3c_6)-9c_5^2}}\,{\rm Vol}(N)|\deg_G(\phi)|.
\end{align*}
The disadvantage of this general bound is that it is not clear what the associated BPS equations are. In general one should expect there to be several BPS equations, in which case finding non-trivial solutions is unlikely. We therefore want to constrain the parameters $c_i$ in such a way for which the number of BPS equations is small; specifically one or two sets of equations. In order for this to be the case, it is straightforward to realise that this forces the energy density for \eqref{gauged-Skyrme-energy} to be of the form
\begin{gather}\label{BPS-bound-step1}
\EE=\kappa\big|\lambda\star_{M}\d^A\phi\mp\phi^{\ast A}\big(\Sigma+3\mu^\sharp\big)\big|^2
 +\big|\phi^{\ast A}(\alpha\Sigma+\beta\mu^\sharp+\gamma\nu)\big|^2\\\nonumber
 \hphantom{\EE=}{}\pm2\kappa\lambda\big\langle \star_{M}\d^A\phi,\phi^{\ast A}\big(\Sigma+3\mu^\sharp\big)\big\rangle ,
\end{gather}
for some constants $\alpha, \beta, \gamma, \kappa, \lambda\in\R$, $\kappa, \lambda>0$ which should be related to the coefficients $c_1,\dots,c_6$. By reorienting and rescaling $(M,g_M)$, we can fix the sign above, and choose $\lambda=1$. We may also fix $\kappa=1$ by fixing an energy scale. So then, from \eqref{charge-density-ip} and \eqref{BPS-bound-step1} we hence have the topological energy bound
\begin{align}
 E[\phi,A]\geq6\,{\rm Vol}(N) |\deg_G(\phi) |,\label{BPS-bound}
\end{align}
which is attained if and only if the BPS equations
\begin{gather}
 \star_{M}\d^A\phi-\phi^{\ast A}\big(\Sigma+3\mu^\sharp\big)=0,\label{BPS1}\\
 \phi^{\ast A}\big(\alpha\Sigma+\beta\mu^\sharp+\gamma\nu\big)=0,\label{BPS2}
\end{gather}
both hold. By comparing coefficients in \eqref{gauged-Skyrme-energy} and \eqref{BPS-bound-step1}, we see that in order for these steps to be valid, the parameters in the energy must relate to the parameters in these equations and bound via{\samepage
\begin{alignat}{4}
 & c_1=1,\qquad && c_2=1+\alpha^2,\qquad&& c_3 =\gamma^2,&\nonumber\\
 & c_4=9+\beta^2,\qquad && c_5 =2\alpha\gamma,\qquad&& c_6 =2(3+\alpha\beta).& \label{parameter-matching}
 \end{alignat}}%
These assignments define a choice of scale on $M$ and the parameters for the energy \eqref{gauged-Skyrme-energy} up to a~constant factor in terms of three free parameters $\alpha, \beta,\gamma\in\R$. In this case, the BPS bound~\eqref{BPS-bound} holds, and is saturated by configurations satisfying the BPS equations \eqref{BPS1} and~\eqref{BPS2}.
The values of these parameters play a key role in understanding the possible solutions of the equations~\eqref{BPS1} and~\eqref{BPS2}.

\subsection[Necessary conditions on the pullback of the moment map]{Necessary conditions on $\boldsymbol{\phi^{\ast A}\mu^\sharp}$}\label{sec:nec-cond-BPS1}

Having established topological energy bounds for the functional \eqref{gauged-Skyrme-energy}, and the BPS equations~\eqref{BPS1} and~\eqref{BPS2}, it is natural to ask if and when solutions exist. While we have not been able to answer this question in full generality, we have been able to classify solutions for specific choices of $N$ and $G$. In this section we shall discuss some general ideas which may be used to solve \eqref{BPS1} and \eqref{BPS2}, which are similar in spirit to the arguments we used in the ungauged case.

Like with \eqref{ungauged bog eq}, it is useful to view the left-hand side of the first of our BPS equation \eqref{BPS1} as the sum of three maps $\Lambda^1N\to\Lambda^2M$. Because of the presence of the third term $3\phi^{\ast A}\mu^\sharp$, there is no a priori reason to only consider the cases where $\d^A\phi$ has zero or full rank; in fact, as we shall see later, there are interesting solutions which occur in other cases which were not possible in the ungauged equations. However, as we show below, the possible ranks of the maps involved are constrained by the rank of $\d^A\phi$. This will allow us to strictly constrain the possible solutions of \eqref{BPS1}.

First consider when $\d^A\phi$ has full rank. This means the map $\phi^{\ast A}\colon \Lambda^1N\to\Lambda^1M$ is invertible, and so we may write \eqref{BPS1} equivalently as
\begin{align}\label{metric-eqn}
 \star_{M}(v)=\phi^{\ast A}\star_{N}\big(\big(\phi^{\ast A}\big)^{-1}(v)\big)+3\phi^{\ast A}(\mu^\sharp)\big(\big(\phi^{\ast A}\big)^{-1}(v)\big)
\end{align}
for all $v\in\Lambda^1M$. Here the action of $3\phi^{\ast A}(\mu^\sharp)$ on $(\phi^{\ast A})^{-1}(v)$ is given by the pairing between $TN$ and $\Lambda^1N$. We aim to view \eqref{metric-eqn} as an equation which determines a metric $g_M$. As it stands,~\eqref{metric-eqn} determines a map $\star_{M}\colon \Lambda^1M\to\Lambda^2M$, not a metric tensor. The space of linear maps $\Lambda_p^1M\to\Lambda_p^2M$ has dimension 9, while the space of metric tensors has dimension 6, so, not every linear map $\Lambda_p^1M\to\Lambda_p^2M$ arises as the Hodge star of a metric. The ones that do, and hence the solutions of \eqref{metric-eqn}, are constrained by the following lemma.
\begin{Lemma}\label{lem:star-constraint}
The Hodge star $\star_{M}\colon \Lambda^1_pM\to\Lambda^2_pM$ of a riemannian metric $g_M$ satisfies
\begin{equation}
\label{star identity}
\tr\bigl(\iota_V\circ\star_M\colon \Lambda^1_pM\to\Lambda^1_pM\bigr)=0\qquad\text{for all }V\in T_pM.
\end{equation}
Furthermore, the map $g_M\mapsto \star_{M}$ is a homeomorphism from the space of riemannian metrics on~$T_pM$ to an open subset of the set of all maps $\star$ satisfying \eqref{star identity}.
\end{Lemma}
\begin{proof}
For the first part, we may choose a basis $e^i$ for $\Lambda^1_pM$ which is oriented and orthonormal with respect to $g_M$. In this basis, $\star_{M}e^i=\frac12 \varepsilon_{ijk}e^j\wedge e^k$ and so letting $E_i$ denote the dual basis for $T_pM$, we have{\samepage
\begin{equation*}
\sum_j \iota_{E_j}\star_{M} e^j = \sum_{j,k,l}\iota_{E_j} \tfrac{1}{2}\epsilon_{jkl}e^k\wedge e^l=\sum_{j,k,l}\big(\delta^{jk}e^l-\delta^{jl}e^k\big)\epsilon_{jkl}=0,
\end{equation*}
as required.}

For the second part, we note that $g_M\mapsto \star_{M}$ is clearly continuous, and that the space of maps~$\star$ satisfying \eqref{star identity} has the same dimension as the space of metric tensors, namely $6$. Therefore, it suffices to exhibit a continuous left inverse $\star\mapsto (g_M)_\star$ of $g_M\mapsto \star_{M}$. A suitable map, written in terms of dual bases $e^i$ and $E_i$, is
\begin{equation}\label{recover-metric}
(g_M)_\star(V,W) = \tfrac12 \sum_{i,j=1}^3 \star e^j(V,E_i)\star e^i(E_j,W)\qquad \text{for all } V,W\in T_pM.
\end{equation}
This is clearly independent of the choice of basis. It is also straightforward to check that if $\star=\star_{M}$ then $(g_M)_\star = g_M$.
\end{proof}
Now consider the case where $\d^A\phi$ is not full rank. These cases are constrained by the following lemma.
\begin{Lemma}\label{lem:nullity}
 If $\rk\big(\d^A\phi\big)<3$, then $\rk\big(\phi^{\ast A}\star_N\big)=\max\bigl\{\rk\big(\d^A\phi\big)-1,0\bigr\}$.
\end{Lemma}
\begin{proof}
 Suppose $u\in\Lambda^1N$ is such that $\phi^{\ast A}u=0$. Then $\phi^{\ast A}(u\wedge v)=0$ for all $v\in\Lambda^1N$. So as $\star$ is an isomorphism the result holds when $\phi^{\ast A}$ has rank $0$ or $1$. In the rank $2$ case, so far all this shows is $\dim\ker\phi^{\ast A}\star_{N}\geq2$. However, by definition, in this case there exists non-zero $e^1, e^2\in\Lambda^1N$ such that $\phi^{\ast A}e^1$, $\phi^{\ast A}e^2$ are linearly independent, and hence $e^1\wedge e^2$ is a non-zero vector not in the kernel of $\phi^{\ast A}\star_{N}$.
\end{proof}
In the particular case, where $\d^A\phi$ has rank $2$, we have a following stronger constraint.
\begin{Lemma}\label{lem:rank-2-reduction}
 If $\d^A\phi$ has rank $2$, then $\star_{M}\phi^{\ast A}+\phi^{\ast A}\star_{N}\colon \Lambda^1N\to\Lambda^2M$ has full rank.
\end{Lemma}
\begin{proof}
First, we show $\d^A\phi$ and $\phi^{\ast A}\star_{N}$ have trivially intersecting kernels. To do this let $u\in\Lambda^1N$ and suppose that $\phi^{\ast A}u=\phi^{\ast A}\star_{N}u=0$. Now, there exist $v,w\in\Lambda^1N$ such that $\star_{N}u=v\wedge w$. Then $\phi^{\ast A}v\wedge\phi^{\ast A}w=0$, so we can find $(a,b)\neq(0,0)$ such that $a\phi^{\ast A}v+b\phi^{\ast A}w=0$. Then $av+bw$ and $u$ are two linearly independent vectors in $\ker\big(\phi^{\ast A}\big)$, contradicting the assumption that $\rk\big(\phi^{\ast A}\big)=2$.

Now let $f^1,f^2\in\Lambda^1N$ be a basis for $\ker\big(\phi^{\ast A}\star_N\big)$. Then $\phi^{\ast A}f^1$, $\phi^{\ast A}f^2$ are linearly independent, and hence form a basis for the image of $\phi^{\ast A}$, because $\ker\big(\phi^{\ast A}\big)\cap\ker\big(\phi^{\ast A}\star_N\big)=\{0\}$. Then the triple
\begin{align*}
\star_{M}\phi^{\ast A}f^1,\qquad\star_{M}\phi^{\ast A}f^2,\qquad\phi^{\ast A}\big(f^1\wedge f^2\big)
\end{align*}
are linearly independent: $\phi^{\ast A}\big(f^1\wedge f^2\big)$ is non-zero because $\phi^{\ast A}f^1$, $\phi^{\ast A}f^2$ are linearly independent, and is orthogonal to both $\star_{M}\phi^{\ast A}f^1$, $\star_{M}\phi^{\ast A}f^2$ by definition of the Hodge star. Now let $f^3=\star_{N}\big(f^1\wedge f^2\big)$. Since $\phi^{\ast A}f^1$, $\phi^{\ast A}f^2$ are a basis for ${\rm im}\big(\phi^{\ast A}\big)$, $\phi^{\ast A}f^3$ is a linear combination of $\phi^{\ast A}f^1$ and $\phi^{\ast A}f^2$. Therefore,
\begin{align*}
& \bigl(\star_{M}\phi^{\ast A}+\phi^{\ast A}\star_{N}\bigr)f^1=\star_{M}\phi^{\ast A}f^1,\\
& \bigl(\star_{M}\phi^{\ast A}+\phi^{\ast A}\star_{N}\bigr)f^2=\star_{M}\phi^{\ast A}f^2,\\
& \bigl(\star_{M}\phi^{\ast A}+\phi^{\ast A}\star_{N}\bigr)f^3=\star_{M}\phi^{\ast A}f^3+\phi^{\ast A}\big(f^1\wedge f^2\big),
\end{align*}
are linearly independent. So $\star_{M}\phi^{\ast A}+\phi^{\ast A}\star_{N}$ has rank $3$, i.e., is full rank.
\end{proof}

Putting this all back into the context of equation \eqref{BPS1}, we can summarise the results in this section as follows.
\begin{Proposition}\label{prop:nec-for-BPS1}
 Let $(\phi,A)$ be a solution of \eqref{BPS1} and let $n$ be the rank of $\d^A\phi$. Then
 \begin{enumerate}\itemsep=0pt
 \item[$1.$] If $n=0$ or $1$, then $\phi^{\ast A}\mu^\sharp$ has rank $n$ and $\ker\phi^{\ast A}\mu^\sharp=\ker\d^A\phi$.
 \item[$2.$] If $n=2$, then $\phi^{\ast A}\mu^\sharp$ is full rank.
 \item[$3.$] If $n=3$, then the linear map
\begin{align*}
 \iota_V\circ\big(\phi^{\ast A}\mu^\sharp\big)\circ\big(\phi^{\ast A}\big)^{-1}\colon \ \Lambda^1M\to\Lambda^1M
\end{align*}
is traceless for all $V\in TM$. Here we regard $\phi^{\ast A}\mu^\sharp\in\Lambda^2M\otimes TN$ as a map $\Lambda^1N\to\Lambda^2M$.
 \end{enumerate}
\end{Proposition}
\begin{proof}
 The first two cases follow immediately from Lemmas \ref{lem:nullity} and \ref{lem:rank-2-reduction} and consistency of~\eqref{BPS1}. For the final case, we apply Lemma \ref{lem:star-constraint}. Since the maps $\star_{M}$ and $\phi^{\ast A}\circ\star_{N}\circ(\phi^{\ast A})^{-1}$ are defined from metrics, they automatically must satisfy the condition \eqref{star identity}, and so consistency of~\eqref{BPS1} requires the stated condition.
\end{proof}

These necessary conditions will allow us to either rule out possible solutions, or to constrain the possibilities for solutions. In particular, in the full rank case, the constraint then allows us to recover a metric $g_M$, as in the proof of Lemma \ref{lem:star-constraint} in equation \eqref{recover-metric}, from \eqref{metric-eqn}. We shall return to these later via some explicit examples.

\section[Solutions with principal orbit S\^{}1]{Solutions with principal orbit $\boldsymbol{S^1}$}\label{sec:princ-orbit-S1}

We now turn our attention solutions of \eqref{BPS1} and~\eqref{BPS2}. The different choices of $G$ and its action on $N$ lead to different solutions. We start our discussion of solutions in this section with the simplest case, where the structure group $G=\U(1)$ acts nontrivially on $N$. This means the target manifold $N$ may be viewed as a fibration over a surface $C$ with typical fibre $S^1$,
\begin{equation*}
\begin{tikzcd}
S^1\arrow{r}{}&N\arrow{d}{}\\
&\hphantom{.}C.
\end{tikzcd}
\end{equation*}
The $S^1$ fibres may collapse above certain points in $C$.

\subsection[Invariant metric on N]{Invariant metric on $\boldsymbol{N}$}

We begin by describing the geometry of $N$ in more detail. The Lie algebra $\uu(1)$ of $\U(1)$ is isomorphic to $\R$, and is generated by $1\in\R$. For brevity, we will write $\nu=\nu(1)$, $\mu=\mu(1)$, $\mu^\sharp=\mu^\sharp(1)$ which generate the Killing field and moment map respectively. Since both $\mu$ and~$g_N$ are $\U(1)$-equivariant, we have that $\big[\nu,\mu^\sharp\big]=0$. Therefore, we can choose local coordinates $(\theta,x,y)$ on $N$ such that
\begin{equation}\label{action-moment-U(1)}
\nu = \frac{\partial}{\partial\theta},\qquad\text{and}\qquad \mu^\sharp=\frac{\partial}{\partial y}.
\end{equation}
The $\U(1)$-invariant metric can be written in the form
\begin{equation}\label{metric-circle-bundle}
g_N = f(x,y)(\d\theta+\omega)^2 + g_C,
\end{equation}
in which the 1-form $\omega$ and 2-tensor $g_C$ satisfy $\omega(\partial_\theta)=0$ and $g_C(\partial_\theta,\cdot)=0$. Invariance requires that the 1-form $\omega$ and symmetric tensor $g_C$ satisfy
\begin{align*}
 \mathcal{L}_{\partial_\theta}\omega=0,\qquad \mathcal{L}_{\partial_\theta}g_C=0.
\end{align*}
Moreover, since $g_N\big(\mu^\sharp,\nu\big)=\mu(\nu)=0$ by \eqref{moment-map-constraint}, it must be that $\omega(\partial_y)=0$. So
\begin{equation*}
\omega=\omega_x(x,y)\,\d x.
\end{equation*}
Note that by \eqref{moment-map-constraint} and \eqref{action-moment-U(1)}, $\mu=\mu_x\,\d x+\mu_y\,\d y=g_N(\bdy_x,\bdy_y)\,\d x+g_N(\bdy_y,\bdy_y)\,\d y$. So from \eqref{metric-circle-bundle} it follows that
\begin{equation*}
g_N = f(x,y)^2(\d\theta+\omega)^2 + \frac{\mu^2}{\mu_y} + h(x,y)\,\d x^2
\end{equation*}
for some function $h(x,y)$. The volume form associated with this metric is
\begin{equation*}
V_N = \sqrt{fh\mu_y}\,\d\theta\wedge\d x \wedge \d y.
\end{equation*}
Therefore, $\iota_\nu V_N = \sqrt{fh\mu_y}\, \d x \wedge \d y$, and by \eqref{moment-map-def} this must equal $\d\mu$. Therefore,
\begin{equation*}
f(x,y) = \frac{(\partial_x\mu_y-\partial_y\mu_x)^2}{h\mu_y}.
\end{equation*}
In summary, the metric on $N$ takes the form
\begin{equation}\label{metric-fixed-U(1)}
g_N = \frac{(\partial_x\mu_y-\partial_y\mu_x)^2}{h\mu_y}(\d\theta+\omega)^2 + h\,\d x^2 + \frac{1}{\mu_y}\mu^2.
\end{equation}

\subsection{Solutions of the BPS equations}

{\bf Solutions of \eqref{BPS1}.} We shall begin the analysis of the BPS equations in this case by first considering all possible solutions $(\phi,A)$ of \eqref{BPS1}. As outlined in Section \ref{sec:nec-cond-BPS1}, it is helpful to consider the possible ranks of $\d^A\phi$.

In the zero rank case, $\d^A\phi=0$ and \eqref{BPS1} implies $F=0$, so the solution is equivalent to the constant solutions in the ungauged model.

Now we argue that there are no solutions in the cases of rank $1$ or $2$. In the case $\rk\big(\d^A\phi\big)=1$, by Lemma \ref{lem:nullity} we have $\phi^{\ast A}\Sigma=0$, and so \eqref{BPS1} may be rearranged to
\begin{align}
\label{S1-rank1-BPS}
 \d^A\phi=3\star_{M}F\otimes\frac{\bdy}{\bdy y}.
\end{align}
This implies that $0=\phi^{\ast A}\d\theta:=\d\theta - A$ (using the definition of $\d^A$ from \eqref{cov-derivative}). So $F=\d A=\d^2\theta=0$, which via \eqref{S1-rank1-BPS} forces $\rk(\d^A\phi)=0$, a contradiction. In the case $\rk(\d^A\phi)=2$, Proposition \ref{prop:nec-for-BPS1} tells us that $\phi^{\ast A}\mu^\sharp$ has rank $3$. But $\phi^{\ast A}\mu^\sharp=F\otimes\tfrac{\bdy}{\bdy y}$ has at most rank $1$, so this case also yields no solutions of \eqref{BPS1}.

It remains to consider the case where $\d^A\phi$ has full rank. Thus $\d^A\phi$ is a linear isomorphism $TM\to TN$. So there exists a unique vector field $V$ on $M$ such that
\begin{equation*}
\d^A\phi(V)=\mu^\sharp=\tfrac{\partial}{\partial y}.
\end{equation*}
Using Proposition \ref{prop:nec-for-BPS1}\,(3), we thus require that
\begin{equation*}
\iota_V F = 0.
\end{equation*}
Writing $\phi$ using our local coordinates on $N$ as $\phi=(\theta,x,y)$, we have
\begin{equation*}
\phi^{\ast A}\d\theta = \d\theta-A,\qquad\phi^{\ast A}\d x = \d x,\qquad\phi^{\ast A}\d y = \d y.
\end{equation*}
It follows that
\begin{equation*}
\iota_V (\d\theta-A) = \d\theta\big(\mu^\sharp\big) = \d\theta(\partial_y)=0.
\end{equation*}
Therefore,
\begin{equation*}
\mathcal{L}_V(\d\theta-A) = \d\iota_V(\d\theta-A) -\iota_V F = 0.
\end{equation*}
So $\d\theta-A$ is invariant under the flow generated by $V$.

We now aim to solve \eqref{BPS1} for the metric $g_M$ (as in \eqref{metric-eqn} and the proof of Lemma \ref{lem:star-constraint}). To do this we fix a convenient gauge. Specifically, we first choose a gauge in which $\iota_V A=0$. Then, since $\iota_VF = 0$, it follows that $\mathcal{L}_VA=0$. So this gauge choice is preserved by gauge transformations $A\mapsto A+\d\lambda$ in which $\mathcal{L}_V\lambda=0$. We may use this gauge freedom to set one of the two remaining components of $A$ to zero. Without loss of generality, we may therefore choose a gauge so that $A=A_x\,\d x$.

In this gauge $\phi^{\ast A} V_N=(\partial_x\mu_y-\partial_y\mu_x)\d\theta\wedge \d x\wedge \d y$. Since $\d^A\phi$ has full rank, this is non-zero, so $\theta$, $x$, and $y$ are a good system of local coordinates on $M$. In these coordinates $V=\frac{\partial}{\partial y}$. The pullback of the metric \eqref{metric-fixed-U(1)} on $N$ is
\begin{equation*}
\phi^{\ast A}g_N = \frac{(\partial_x\mu_y-\partial_y\mu_x)^2}{h\mu_y}(\d\theta+(\omega_x-A_x)\d x)^2 + h\,\d x^2 + \frac{\mu^2}{\mu_y}.
\end{equation*}
Solving \eqref{BPS1} for $g_M$ then gives
\begin{equation}\label{metric-sol-U1}
g_M = \left(1+3\frac{F_{\theta x}\mu_y}{(\partial_x\mu_y-\partial_y\mu_x)}\right)\left(\frac{(\partial_x\mu_y-\partial_y\mu_x)^2}{h\mu_y}(\d\theta+\omega-A)^2 + h\,\d x^2\right) + \frac{\mu^2}{\mu_y}.
\end{equation}

{\bf Solutions of \eqref{BPS2}.} We now address the second BPS equation \eqref{BPS2}. Clearly these are solved when $\alpha=\beta=\gamma=0$. As we shall see, in general this is the only non-trivial case compatible with the solutions of \eqref{BPS1}.

Suppose $\alpha=0$ and $\beta$, $\gamma$ are not both $0$. Since $\nu$ and $\mu^\sharp$ are linearly independent, \eqref{BPS2} then implies that $F=0$. This is compatible with the solutions of \eqref{BPS1} found in the previous subsection, however they are just solutions of the ungauged Skyrme model, i.e., $\phi$ is locally an~isometry or is (covariantly) constant.

When $\alpha\neq0$, again by linear-independence of $\nu$ and $\mu^\sharp$, \eqref{BPS2} implies that $\phi^{\ast A}\Sigma$ has rank~$1$ or~$0$. However, as we saw above, \eqref{BPS1} is inconsistent when $\d^A\phi$ has rank $1$ or $2$, so by Lemma~\ref{lem:nullity} the only possibility is the trivial solutions where $F=0$ and $\d^A\phi=0$.

Summarising, we have proved the following statement.

\begin{Theorem}\label{thm:classification-S1}
If $G=\U(1)$ acts non-trivially on $N$, then the BPS equations \eqref{BPS1} and~\eqref{BPS2} have solutions with $F\neq0$ only when $\alpha=\beta=\gamma=0$. When $\alpha=\beta=\gamma=0$, the solutions consist (locally) of the identity map $\phi\colon M=N\to N$, a choice of connection $A$ such that $\iota_\nu A = \iota_{\mu^\sharp}A=\mathcal{L}_{\mu^\sharp}A=0$,
and a metric $g_M$ given by
\begin{equation*}
g_M = g_\mu + (1+3\star_N (F\wedge\mu))g_C,
\end{equation*}
in which $g_\mu=\mu^2/\mu\big(\mu^\sharp\big)$ and $g_C=\phi^{\ast A}g_N-g_\mu$.
\end{Theorem}
The theorem only says that $\phi$ is the identity map locally; it may not even be a bijection globally (for example $N$ may be a quotient of $M$ by a discrete group). However, if $\phi\colon N\to N$ is globally just the identity map then its topological degree is 1. This is because
\begin{equation*}
\int_N\phi^{\ast A}(V_N+\mu) = \int_N\big( V_N - A\iota_\nu V_N + \d A\wedge \mu\big) = \int_N\big(V_N + \d(A\wedge\mu)\big) = \int_N V_N.
\end{equation*}

\subsection{Example}

Consider the adjoint action of $\U(1)$ on $N=\SU(2)$. We identify $\SU(2)$ with $S^3$ and choose coordinates $(\theta,x,y)\in(0,2\pi)\times \big(0,\frac{\pi}{2}\big)\times (0,4\pi)$ such that
\begin{equation*}
\phi=\bigl( \sin x\cos\tfrac{y}{2},\cos x\cos \theta,\cos x\sin\theta , \sin x\sin\tfrac{y}{2}\bigr)
\end{equation*}
and $\nu=\frac{\partial}{\partial \theta}$. The round metric on $S^3$ in these coordinates is
\begin{equation*}
g_N = \cos^2(x)\,\d\theta^2 + \d x^2 + \tfrac{1}{4}\sin^2 (x)\,\d y^2.
\end{equation*}
Then $V_N=\frac12\sin (x)\cos (x) \d\theta\wedge \d x\wedge \d y$ and $\iota_\nu V_N =\d\mu$, where $\mu=\frac{1}{4}\sin^2 (x)\,\d y$. This choice of~$\mu$ leads to $\mu^\sharp=\frac{\partial}{\partial y}$, consistent with our earlier conventions.

When $\alpha=\beta=\gamma=0$, the energy function \eqref{gauged-Skyrme-energy} with BPS coefficients \eqref{parameter-matching} is given by
\begin{equation*}
E = \int_M \big(\big|L^A\big|^2 + \big|L^A\wedge L^A-\tfrac{3}{4}U^{-1}\{{\rm i}\sigma_3,U\}F\big|^2\big),
\end{equation*}
in which $L^A=U^{-1}\bigl(\d U + \tfrac{\rm i}{2}A[\sigma_3,U]\bigr)$, $U=\phi_4-{\rm i}\sum_{j=1}^3\phi_j\sigma_j$ and $\sigma_j$ are the Pauli matrices. This shows that the model includes a Yang--Mills term, but the coefficient of this term depends on~$U$.

Following the discussion above, solutions of \eqref{BPS1} are given by the identity map $\phi\colon S^3\to S^3$ and a gauge field $A=A(x,\theta)\,\d x$.
The metric on $M$ should then be chosen to be as in \eqref{metric-sol-U1}, which here takes the form
\begin{equation*}
g_M = \bigl(1+\tfrac{3}{2}{\partial_\theta A_x}\tan x\bigr)\bigl(\cos^2(x)\,\d\theta^2 + \d x^2\bigr) + \tfrac14\sin^2 (x)\,\d y^2.
\end{equation*}

\section[Solutions with principal orbit S\^{}2]{Solutions with principal orbit $\boldsymbol{S^2}$}\label{sec:princ-orbit-S2}

We now look to cases with $G=\SU(2)$. In this section, we consider where the stabiliser of a~typical point in $N$ is $\U(1)$. Then the principal orbits of the $\SU(2)$-action are $S^2\cong\SU(2)/\U(1)$, and $N$ is a compactification of the manifold
\begin{gather*}
 \II\times S^2,
\end{gather*}
where $\II$ is an interval. We begin our discussion by describing the geometry of $N$ in more detail.

\subsection[Geometry of I times S\^{}2]{Geometry of $\boldsymbol{\II\times S^2}$}

We model $S^2$ as the set of unit length elements of the Lie algebra $\su(2)$. More precisely, we choose the inner product $(X,Y)=-\tfrac{1}{2}\tr(XY)$ for $X, Y\in\su(2)$, and choose the orthonormal basis defined by Pauli matrices $-\ii\sigma_j$. Then $S^2$ is the set of $x=-\ii x^j\sigma_j\in\su(2)$ such that $(x,x)=1$. The metric on $S^2$ is induced from that on $\su(2)$ and can be written as
\begin{equation*}
g_{S^2}=(\d x,\d x)=\d x^j\,\d x^j.
\end{equation*}
The corresponding area 2-form is
\begin{align*}
\omega_{S^2}=-\tfrac{1}{4}\tr(x\,\d x\wedge\d x)=\tfrac{1}{2}\epsilon_{ijk}\, x^i\d x^j\wedge\d x^k.
\end{align*}
The action of $g\in \SU(2)$ is $g\cdot x = gxg^{-1}$. By choosing a suitable coordinate $\xi$ on $\II$, the most general metric on $N$ invariant under this action may be written as
\begin{align}
 g_N=h_1(\xi)^2\,\d\xi^2+h_2(\xi)^2\,g_{S^2},\label{metric-N-S2}
\end{align}
where $h_i\colon \II\to\R_{>0}$ are smooth functions. The associated volume form is
\begin{align}
\label{adjoint-nu}
 V_N=h_1(\xi)h_2(\xi)^2\,\d\xi\wedge\omega_{S^2}.
\end{align}
The section $\Sigma$ of $\Lambda^2 N\otimes TN$ that represents the Hodge star of the metric \eqref{metric-N-S2} is given by
\begin{align}
\label{adjoint-Sigma}
\Sigma = \star_{N} \d \xi\otimes\frac{\bdy}{\bdy \xi} + \star_{N}\d x = \frac{h_2(\xi)^2}{h_1(\xi)}\,\omega_{S^2}\otimes\frac{\bdy}{\bdy \xi} + h_1(\xi)\,\d\xi\wedge x\, \d x.
\end{align}
In this equation, we are representing tangent vectors to $S^2$ as $\su(2)$ matrices $X$ satisfying $(X,x)=0$. Equivalently, we are identifying $-\ii\sigma_j$ with $\bdy/\bdy x^j$.

From \eqref{nu-definition}, the linear map $\nu\colon \su(2)\to \Gamma(TN)$ describing the adjoint action is
\begin{align}
 \nu(X)=\frac{\d}{\d t}\Big|_{t=0}\exp(-tX)x\exp(tX)=[x,X].\label{adjoint-action}
\end{align}
It follows that
\begin{align*}
 \iota_{\nu(X)}V_N=-\tfrac{1}{2}h_1(\xi)h_2(\xi)^2\,\d\xi \wedge\tr\bigl(x\,\d x\, [x,X]\bigr) = 2h_1(\xi)h_2(\xi)^2\,\d\xi \wedge (\d x,X),
\end{align*}
where the second equality makes use of the identities $x^2=-1$ and $x\,\d x = -\d x\, x$. Finally, to determine the moment map $\mu\colon \g\to\Omega^1(N)$, we start by writing the most general form of such a map with the assumption of $\SU(2)$-equivariance. This is
\begin{align}\label{moment-map-adjoint}
 \mu(X)(\xi,x)=\eta_1(\xi)\,\d\xi\,(X,x)+\eta_2(\xi)\,(\d x,X)+\eta_3(\xi)\,\bigl(\d x,[X,x]\bigr),
\end{align}
where $\eta_1, \eta_2, \eta_3\colon \II\to\R$ are smooth. Imposing the constraint $\d\mu(X)=\iota_{\nu(X)}V_N$ forces
\begin{align}\label{moment-map-adjoint-cond}
 2h_1(\xi)h_2(\xi)^2=\eta_2'(\xi)-\eta_1(\xi),\qquad\eta_3(\xi)=0.
\end{align}
It is straightforwardly verified that these constraints on $\mu$ also allow for $\iota_{\nu(X)}\mu(X)=0$, and so in this case all conditions are satisfied so that the equivariant topological degree \eqref{equivariant-top-degree} is a~well-defined invariant for this model.

Using the metric $g_N$, we hence find
\begin{align}\label{sharped-moment-adjoint}
 \mu^\sharp(X)=\frac{\eta_1(\xi)}{h_1(\xi)^2}\,(X,x)\,
 \frac{\bdy}{\bdy\xi}+\frac{\eta_2(\xi)}{h_2(\xi)^2}\,\bigl(X-(X,x)x\bigr).
\end{align}
Altogether the metric and moment map depend on four functions $\eta_1$, $\eta_2$, $h_1$, $h_2$. These are subject to one constraint \eqref{moment-map-adjoint-cond}, and can further be fixed by choice of the coordinate $\xi$. For example, we could choose coordinates in which $h_1(\xi)=1$, and take \eqref{moment-map-adjoint-cond} as the definition of $h_2$, written in terms of two functions $\eta_1$, $\eta_2$ satisfying $\eta_2'-\eta_1>0$. So altogether, the metric on $N$ and the moment map $\mu$ combine to give two functional degrees of freedom from which one can determine the other objects in the BPS equations.

There are two natural ways to obtain a compact manifold $N$ (without boundary) from \mbox{$\II=(0,T)$}: we can either identify $(0,x)$ with $(T,x)$ to obtain $N=S^1\times S^2$, or we can collapse the spheres $\{0\}\times S^2$ and $\{T\}\times S^2$ to points, obtaining $N=S^3$. In the first case, we require that $\eta_1$, $\eta_2$, $h_1$, $h_2$ extend to periodic functions of $\xi\in\R$, so that $g_{S^2}$ and $\mu$ extend smoothly to the compactification. In the second case, working in coordinates where $h_1(\xi)=1$, we require that $h_2(\xi), \eta_1'(\xi), \eta_2(\xi)\to 0$ as $\xi$ tends to $0$ or $T$. In either case, $\int_\II \eta_2'\,\d\xi=0$ and so
\begin{equation}\label{eta1 constraint}
\int_\II \eta_1\,\d\xi = -\int_\II 2 h_1h_2^2\,\d\xi=-\operatorname{Vol}(N)/2\pi<0.
\end{equation}
In particular, $\eta_1$ is non-zero.
\begin{Remark}
In the case $h_1(\xi)=1$, $h_2(\xi)=\sin\xi$, $(N,g_N)$ is the round three-sphere $S^3\equiv\SU(2)$. Using the diffeomorphism $U\colon (0,\pi)\times S^2\to\SU(2)$ given by
\begin{align*}
 U(\xi,x)=\cos\xi+\sin\xi\,x,
\end{align*}
we can reconcile the moment map \eqref{moment-map-adjoint} with the choice identified previously in \eqref{moment-map-adjoint-round-sphere-choice} by setting
\begin{align*}
 \eta_1(\xi)=-1,\qquad\eta_2(\xi)=-\tfrac{1}{2}\sin2\xi.
\end{align*}
It is easily checked that these functions satisfy \eqref{moment-map-adjoint-cond}.
\end{Remark}

Having described the geometry of $N$, we now identify some natural solutions of the BPS equations \eqref{BPS2} and \eqref{BPS1}. Recall that, locally, $\phi$ is a function $M\to N$. Using our model of~$N$ as a product of $\II$ and $S^2\subset \su(2)$, this corresponds to a pair of functions $\xi\colon M\to \II$ and $\Phi\colon M\to\su(2)$, such that $\Phi^2=-1$. Globally, we may view $\Phi$ as a section of $\operatorname{End}(E)$, where $E\to M$ is the vector bundle $P\times_{\SU(2)}\C^2$. Then there is a natural splitting $E=L\oplus L^\ast$, where~$L$ and~$L^\ast$ are the eigenbundles associated to the eigenvalues ${\rm i}$ and~$-{\rm i}$ of $\Phi$, respectively.
\subsection{Dirac monopoles}\label{sec:Dirac monopoles}
The simplest solutions that we have identified are where $\Phi$ is covariantly constant, by which we mean that $\d^{A}\Phi=0$. In this case, \smash{$\d^A\phi=\d\xi\frac{\bdy}{\bdy\xi}$} has rank~1, and so by Lemma \ref{lem:nullity}, $\Phi^{\ast A}\Sigma=0$.
We choose a local gauge in which~$\Phi$ is constant.
Since $\Phi$ is parallel, the connection $A$ restricts to a~connection $\ii a$ on $L$, and we can write $A=\Phi\,a$ and $F=\Phi\,\d a$. Then, from \eqref{adjoint-action} and \eqref{sharped-moment-adjoint},
\begin{gather*}
\phi^{\ast A}\nu = [\Phi,F] = 0\qquad\text{and}\\
\phi^{\ast A}\mu^\sharp = \frac{\eta_1(\xi)}{h_1(\xi)^2}\,(F,\Phi)\otimes \frac{\bdy}{\bdy\xi} + \frac{\eta_2(\xi)}{h_2(\xi)^2}\,\bigl(F-(F,\Phi)\Phi\bigr)
= \frac{\eta_1(\xi)}{h_1(\xi)^2}\,\d a\otimes \frac{\bdy}{\bdy\xi}.
\end{gather*}
The BPS equations \eqref{BPS1} and \eqref{BPS2} are simply
\begin{gather*}
\beta \frac{\eta_1(\xi)}{h_1(\xi)^2}\,\d a\otimes \frac{\bdy}{\bdy\xi}=0,\\
\star_{M} \frac{h_1(\xi)^2\d \xi}{3\eta_1(\xi)}\otimes \frac{\bdy}{\bdy\xi} = \d a\otimes \frac{\bdy}{\bdy\xi}.
\end{gather*}
In order to obtain non-trivial solutions, we choose $\beta=0$. Then the first equation is trivially satisfied. In coordinates where $\eta_1(\xi)=-\frac{h_1(\xi)^2}{3}$, the second is the BPS equation $\star_{M}\d \xi=-\d a$ for abelian monopoles. The most well-known solution is the Dirac monopole, for which $a$ is a~connection on a degree 1 line bundle over $\R^3\setminus\{0\}$ and $\xi=\tfrac{1}{2r}$. Similar solutions with point-like singularities can be found on other manifolds $(M,g_M)$ (see, for example, \cite{biswashurtubise2015monopoles}).

\subsection{Solutions from spinor bundles}
\label{sec:spinorial}

For our next family of solutions, we choose a Riemann surface $C$ with local holomorphic coordinate $z$ and metric $g_C=\Omega(z,\ol{z})\,\d z\d\ol{z}$. Let $S\to C$ be a spinor bundle for $C$, which is a~vector bundle of rank 2. Let $\Phi\colon S\to S$ be given by Clifford multiplication with the volume form $\omega_C=\frac{{\rm i}}{2}\Omega\,\d z\wedge \d\ol{z}$. The eigenspaces of $\Phi$ give two subbundles $S^\pm$ such that $S=S^-\oplus S^+$ and $\Phi s=\pm{\rm i} s$ for $s\in S^\pm$.

The Levi-Civita connection on $TC$ induces a natural connection $\nabla^{\rm LC}$ on $S$. There is a second natural connection $\nabla^A$ on $S$ defined as follows
\begin{equation}
\label{baer connection}
\nabla^A_Vs = \nabla^{\rm LC}_Vs + \tfrac{1}{2}V\cdot s\qquad \text{for all } s\in\Gamma(S),\,V\in\Gamma(TC)
\end{equation}
(this connection previously appeared in the classification of riemannian manifolds admitting real Killing spinors \cite{Bar1993}). The curvature of this connection is given in terms of the Gauss curvature~$K$~by
\begin{equation}
\label{spinor curvature}
F^A=\tfrac{1}{2}(1-K)\omega_C\otimes\Phi.
\end{equation}
This identity is easily verified by direct calculation. First, the Levi-Civita connection coincides with the Chern connection, so is given by $\nabla^{\rm LC}\frac{\partial}{\partial z} = (\partial \ln\Omega)\otimes \frac{\partial}{\partial z}$.
Therefore,
\[
\nabla^{\rm LC}\left(\Omega^{-\frac12}\frac{\partial}{\partial z}\right)=\tfrac{1}{2}\bigl(\partial-\bar\partial\bigr)\ln\Omega \otimes\frac{\partial}{\partial z}.
\]
We choose a representation of the Clifford algebra in which
\begin{equation*}
\Omega^{-\frac12}\frac{\partial}{\partial z} \mapsto \begin{pmatrix}0&0\\1&0\end{pmatrix}.
\end{equation*}
Then
\begin{equation*}
\Omega^{-\frac12}\frac{\partial}{\partial \ol{z}} \mapsto -\begin{pmatrix}0&0\\1&0\end{pmatrix}^\dagger = \begin{pmatrix}0&-1\\0&\hphantom{-}0\end{pmatrix}.
\end{equation*}
We note that Clifford multiplication with $\Omega^{\frac12}\,\d z$ is the same as Clifford multiplication with $2\Omega^{-\frac12}\frac{\partial}{\partial \ol{z}}$, because they are related by the musical isomorphism. Therefore,
\begin{equation*}
\omega_C\cdot s = \frac{{\rm i}}{4}\left[\Omega^{\frac12} \d z,\Omega^{\frac12} \d\ol{z}\right]\cdot s = {\rm i}\left[\Omega^{-\frac12}\frac{\partial}{\partial\ol{z}},\Omega^{-\frac12}\frac{\partial}{\partial z}\right]\cdot s
\end{equation*}
for all $s\in S$, and hence
\begin{equation}
\label{spinor Phi}
\Phi = \begin{pmatrix}-{\rm i}&0\\\hphantom{-}0&{\rm i}\end{pmatrix}.
\end{equation}
In this representation, the Levi-Civita connection takes the form
\begin{equation*}
\nabla^{\rm LC}=\d+\begin{pmatrix}-{\rm i} a & 0 \\ \hphantom{-}0 & {\rm i} a\end{pmatrix}.
\end{equation*}
for a 1-form $a$. Then
\begin{equation*}
\left[ \nabla^{\rm LC},\begin{pmatrix}0&0\\1&0\end{pmatrix}\right] = 2{\rm i} a\otimes \begin{pmatrix}0&0\\1&0\end{pmatrix}.
\end{equation*}
Comparing with our earlier calculation, we deduce that ${\rm i} a = \frac{1}{4}(\partial-\bar{\partial})\ln\Omega$. So the connection matrix of $\nabla^A$ in this gauge is
\begin{equation*}
A = \begin{pmatrix} \frac{1}{4}(\bar{\partial}-\partial)\ln\Omega & -\frac{1}{2}\Omega^{\frac12}\d\ol{z} \\ \frac{1}{2}\Omega^{\frac12}\d z & \frac{1}{4}(\partial-\bar{\partial})\ln\Omega\end{pmatrix}.
\end{equation*}
Its curvature is
\begin{equation}
\label{spinor F}
F=\d A+A\wedge A = \begin{pmatrix} \frac{{\rm i}}{2}\rho -\frac{{\rm i}}{2}\omega_C & 0 \\ 0 & -\frac{{\rm i}}{2}\rho +\frac{{\rm i}}{2}\omega_C\end{pmatrix},
\end{equation}
in which $\rho=-{\rm i}\partial\bar{\partial}\ln\Omega=K\omega_C$ is the Ricci form. This completes the proof of equation \eqref{spinor curvature}.

This explicit calculation also shows that
\begin{equation}
\label{spinorial dPhi}
\d^A\Phi = \begin{pmatrix}0&-{\rm i}\Omega^{\frac12}\d\ol{z} \\-{\rm i}\Omega^{\frac12}\d z&0\end{pmatrix},
\end{equation}
from which it follows that
\begin{gather*}
\Phi^{\ast A}g_{S^2} = -\tfrac12\tr\bigl(\d^A\Phi\,\d^A\Phi\bigr)=\Omega\,\d z\d\ol{z}=g_{C},\\
\Phi^{\ast A}\omega_{S^2} = -\tfrac14\tr\bigl(\Phi\,\d^A\Phi\wedge \d^A\Phi\bigr)=\tfrac{{\rm i}}{2}\Omega\,\d z\wedge\d\ol{z}=\omega_{C}.
\end{gather*}
Returning to the BPS equations, we choose $M=\II\times C$ and choose $E=S$ to be the spinor bundle of $C$, pulled back to $\II$ via the obvious projection. The section $\Phi$ and connection $A$ are those defined above (but pulled back to $\II\times C$). The function $\xi\colon \II\times C\to \II$ is just projection onto the first factor.

From equation \eqref{spinor curvature} (or, equivalently, equations \eqref{spinor Phi} and \eqref{spinor F}), we find that
\begin{equation*}
\phi^{\ast A}\nu=\bigl[\Phi,F^A\bigr]=0.
\end{equation*}
We will assume that $\alpha=\beta=0$, then \eqref{BPS2} is solved. Turning to \eqref{BPS1}, we choose the coordinate $\xi$ on $N$ so that $h_1=1$ and from \eqref{adjoint-Sigma} obtain
\begin{equation}
\label{spinorial-Sigma}
\phi^{\ast A}\Sigma = h_2^{\,2}\,\omega_C\frac{\partial}{\partial \xi} + \d\xi\wedge\Phi\,\d^A\Phi.
\end{equation}
On the other hand, from \eqref{spinor curvature}, and the fact that $(\Phi,\Phi)=1$, we obtain
\begin{equation}
\label{spinorial-musharp}
\phi^{\ast A}\mu^\sharp = \tfrac{1}{2}\eta_1(1-K)\omega_C\frac{\partial}{\partial \xi}.
\end{equation}
We make an ansatz for the metric on $M$ of the form
\begin{equation*}
g_M = \d\xi^2 + f(\xi)g_C,
\end{equation*}
so that $\star_{M} \d\xi = f(\xi)\omega_C$ and
\begin{equation*}
\star_{M} \d^A\Phi = \star_{M} \begin{pmatrix} 0&-\frac{{\rm i}}{2}\Omega^{\frac12}\d\ol{z} \\ -\frac{{\rm i}}{2}\Omega^{\frac12}\d z & 0 \end{pmatrix} = \begin{pmatrix} 0&\frac{1}{2}\Omega^{\frac12}\d\ol{z} \\ -\frac{1}{2}\Omega^{\frac12}\d z & 0 \end{pmatrix} = \d\xi\wedge\Phi\, \d^A\Phi.
\end{equation*}
So
\begin{equation*}
\star_{M} \d^A\phi = f(\xi)\omega_C + \d\xi\wedge\Phi\, \d^A\Phi.
\end{equation*}
Combining this with \eqref{spinorial-Sigma} and \eqref{spinorial-musharp}, we see that \eqref{BPS1} (or equivalently \eqref{metric-eqn}) is solved by
\begin{equation}
\label{spinorial-metric}
g_M = \d\xi^2 + \bigl(h_2(\xi)^{2}+\tfrac32\eta_1(\xi)(1-K)\bigr)g_C.
\end{equation}
This may or may not be a riemannian metric, depending on the sign of the coefficient of $g_C$.

Note that $g_M$ is in general \emph{not} the pullback of the metric on $N$. The latter is given by
\begin{equation*}
\phi^{\ast A}g_N = \d\xi^2 + h_2(\xi)^{2}g_C.
\end{equation*}
Since $\eta_1$ is a non-vanishing function, $g_M=\phi^{\ast A} g_N$ only when $K$ is constant and equal to 1, and~$g_C$ is the metric on the sphere of unit radius. In this case $A$ is flat and $\phi$ solves the BPS equation of the ungauged Skyrme model.

We now consider the global topology of our solutions. Suppose that $C$ is compact and that~$N$ is compact. Recalling that we have chosen $h_1(\xi)=1$, \eqref{eta1 constraint} gives
\begin{equation*}
\begin{aligned}
\operatorname{Vol}(M)&=
\int_{\II\times C}\bigl(h_2(\xi)^{2}+\tfrac32\eta_1(\xi)(1-K)\bigr)\,\d\xi\wedge\omega_C \\
&= \int_{\II\times C}\bigl(h_2(\xi)^{2}-3h_2(\xi)^2(1-K)\bigr)\,\d\xi\wedge\omega_C \\
&= \int_{\II}h_2(\xi)^{2}\,\d\xi\bigl(6\pi\chi(C)-2\operatorname{Area}(C)\bigr).
\end{aligned}
\end{equation*}
If \eqref{spinorial-metric} is a riemannian metric, then this integral must be positive. Therefore, this construction gives solutions with $M$ and $N$ compact and riemannian only when the Euler characteristic satisfies $\chi(C)>0$, i.e., when $C$ is a sphere. This means that $M$ and $N$ are homeomorphic to either $S^1\times S^2$ or $S^3$ (or a quotient of one of these spaces).

Interestingly, $M$ and $N$ need not be homeomorphic. To see this, consider the case where $g_N$ is the round metric on $N=S^3$. Recall that $h_2(\xi)=\sin(\xi)$, and we may choose $\eta_1(\xi)=-1$. If the Gauss curvature $K$ of $C=S^2$ is greater than 1, then the metric \eqref{spinorial-metric} extends to $S^1\times S^2$, which is clearly not homeomorphic to $S^3$.

Continuing our discussion of the global topology, we consider the topological degree \eqref{equivariant-top-degree} of~our solutions, assuming that $M$ and $N$ are compact. The volume of $N$ is
\begin{align*}
\operatorname{Vol}(N)=4\pi\int_\II h_2(\xi)^2\,\d\xi.
\end{align*}
We need to compare this with the integral of $\phi^{\ast A}(V_N+\mu)$ over $M$. From \eqref{moment-map-adjoint}, we find that
\begin{equation*}
\begin{aligned}
\phi^{\ast A}V_N &= h_2(\xi)^2\,\d\xi\wedge\omega_C \qquad\text{and}\qquad
\phi^{\ast A}\mu &= \tfrac{1}{2}(1-K)\eta_1(\xi)\,\d\xi\wedge\omega_C.
\end{aligned}
\end{equation*}
Since $\eta_1=\eta_2'-2h_2(\xi)^2$,
\begin{equation*}
\int_M \phi^{\ast A}\mu = \int_\II \bigl(\tfrac12\eta_2'-h_2(\xi)^2\bigr)\,\d\xi\int_C(1-K)\omega_C= \int_\II h_2(\xi)^2 \,\d\xi\left(2\pi\chi(C)-\int_C\omega_C\right),
\end{equation*}
where we made use of the boundary conditions for $\eta_2$ discussed earlier. Therefore,
\begin{align*}
\int_M \phi^{\ast A}(V_N+\mu) = 2\pi\chi(C)\int_\II h_2(\xi)^2\,\d\xi = \frac{\chi(C)}{2}\operatorname{Vol}(N).
\end{align*}
So the topological degree of these solutions is half the Euler characteristic of $C$, or equivalently one minus the genus of $C$. Note that this is positive only when $C=S^2$, which is consistent with the fact that $g_M$ is riemannian only in this case.

\subsection{Twisted spinorial solutions}
\label{sec:twisted}
Now we consider a small variation on the spinorial solutions discussed above. For horizontal vectors $V\in TC$ we let $\nabla^A_V$ be given by the same formula \eqref{baer connection}, but let
\begin{equation*}
\nabla_{\partial_\xi}s = B\Phi\cdot s\qquad\text{for all }s\in\Gamma(S)
\end{equation*}
for some constant $B\in\R$. In other words, $A$ differs from the pullback of the connection \eqref{baer connection} by a 1-form $B\Phi\,\d\xi$. Going through the calculations shows that $\d^A\Phi$ is still given by \eqref{spinorial dPhi}, but $F$ is now given by
\begin{equation*}
F = \tfrac{1}{2}(1-K)\omega_C\otimes\Phi-B\,\d\xi\wedge\d^A\Phi.
\end{equation*}
We find that
\begin{gather*}
\phi^{\ast A}\nu = -2B\,\d\xi\wedge\Phi\,\d^A\Phi, \\
\phi^{\ast A}\mu^\sharp = \frac{\eta_1}{2}(1-K)\omega_C\frac{\partial}{\partial\xi} - \frac{B\eta_2}{h_2^{\,2}}\,\d\xi\wedge\d^A\Phi,
\end{gather*}
while $\phi^{\ast A}\Sigma$ is unchanged from \eqref{spinorial-Sigma}. So the BPS equation \eqref{BPS2} is equivalent to
\begin{gather*}
\alpha h_2(\xi)^2 + \tfrac{1}{2}\beta\eta_1(\xi)(1-K) =0,\\
2\gamma B-\alpha =0, \\
B\beta\eta_2(\xi) =0.
\end{gather*}
We will assume that $h_1(\xi)=1$, $\eta_2(\xi)=0$ and $\gamma\neq0$. Then the third equation is satisfied, and the second equation is solved by $B=\frac{\alpha}{2\gamma}$. If $\alpha\neq0$, then the first equation implies that $h_2(\xi)^2=\beta\eta_1(\xi)(K-1)/2\alpha$. Since $h_2\neq0$ this means that $\beta\neq0$. Note that $\beta$ can be non-zero only if the Gauss curvature $K$ of $C$ is constant (as $\alpha h_2(\xi)^2$ is independent of the coordinate on~$C$).\looseness=-1

Since we have assumed that $\eta_2=0$, $\phi^{\ast A}\mu^\sharp$ is given by the same formula \eqref{spinorial-musharp} as in the previous calculation, and the BPS equation \eqref{BPS1} is once again solved by choosing the metric~\eqref{spinorial-metric}.

\subsection{Spherical solutions}\label{sec:spherical}

For our next family of solutions, we choose $M=\II\times S^2$ and choose $\phi\colon M\to N$ to be the identity map. In particular, $\xi\colon \II\to\II$ is the identity map and $\Phi(x)=x$. Our connection $A$ will be spherically-symmetric and take the following form
\begin{equation}
\label{spherical-ansatz}
A = \tfrac12(f(\xi)-1) x\,\d x.
\end{equation}
Then
\begin{align*}
F = \tfrac14\bigl(f^2-1\bigr)\,\d x\wedge \d x +\tfrac{1}{2}f'\,\d\xi\wedge x\,\d x
\qquad\text{and}\qquad
\d^A\Phi = f\,\d x.
\end{align*}
It follows that
\begin{align*}
\Phi^{\ast A}\omega_{S^2}=-\tfrac14\tr\bigl(\Phi\,\d^A \Phi\wedge\d^A\Phi\bigr) = f^2\omega_{S^2}.
\end{align*}
Thus, from \eqref{adjoint-nu}, \eqref{adjoint-Sigma} and \eqref{moment-map-adjoint},
\begin{gather*}
\phi^{\ast A}\Sigma = \frac{h_2^{\,2}f^2}{h_1}\,\omega_{S^2}\otimes\frac{\bdy}{\bdy\xi}+h_1 f\,\d\xi\wedge x\,\d x, \\
\phi^{\ast A}\mu^\sharp = \frac{\eta_1\big(f^2-1\big)}{2h_1^{\,2}}\omega_{S^2}\otimes\frac{\bdy}{\bdy\xi} + \frac{\eta_2f'}{2h_2^{\,2}}\,\d\xi\wedge x\,\d x, \\
\phi^{\ast A}\nu = \tfrac{1}{2}f'\,\d\xi\wedge \d x.
\end{gather*}
The BPS equation \eqref{BPS2} is then equivalent to
\begin{align}
&2\alpha h_1h_2^{\,2}f^2+\beta\eta_1\big(f^2-1\big)=0, \label{spherical-BPS2a} \\
&2\alpha h_1h_2^{\,2}f+\beta\eta_2f'=0, \label{spherical-BPS2b}\\
&\gamma f' =0.\nonumber
\end{align}
If $\gamma\neq0$, the equations imply that $f$ is constant. Also if $\alpha\neq0$ and $\beta=0$, then $f=0$. These solutions both correspond to a special case of the spinorial solutions. If $\alpha=\gamma=0$ and $\beta\neq0$, they imply that $f^2=1$ and hence that $F=0$. So we consider instead the case where $\gamma=0$ and $\alpha,\beta\neq0$. In this case, the two equations \eqref{spherical-BPS2a} and \eqref{spherical-BPS2b} need to be solved in conjunction with~\eqref{moment-map-adjoint-cond}. These three equations imply that
\begin{gather*}
\frac{ff'}{f^2-1}-\frac{\eta_1}{\eta_2} =0, \\
\frac{\beta f'}{\alpha f} - \frac{\eta_1}{\eta_2} +\frac{\eta_2'}{\eta_2}=0.
\end{gather*}
We choose the coordinate $\xi$ on $N$ such that $\eta_2/\eta_1=\xi$. Then these two equations are solved by
\begin{align*}
f(\xi)= \sqrt{1+C_1\,\xi^2},\qquad
\eta_2(\xi) = C_2\,\xi f(\xi)^{-\frac{\beta}{\alpha}}
\end{align*}
for constants $C_1$ and $C_2$. The full system \eqref{moment-map-adjoint-cond}, \eqref{spherical-BPS2a} and \eqref{spherical-BPS2b} is then solved by
\begin{gather*}
\eta_1(\xi) = C_2\,f(\xi)^{-\frac{\beta}{\alpha}},\nonumber \\
h_1(\xi)h_2(\xi)^2 = -\frac{\beta}{2\alpha}C_1C_2\,\xi^2 f(\xi)^{-\frac{\beta}{\alpha}-2}.
\end{gather*}
This completes the solution. We note that the equations do not constrain the metric functions $h_1$, $h_2$, but only the combination $h_1h_2^{\,2}$. We also make note of one important solution: if $C_1=1$, $C_2=-1$, $\beta=2\alpha$ and $h_1=1/\big(1+\xi^2\big)$, then \eqref{spinorial-metric} and \eqref{moment-map-adjoint} are{\samepage
\begin{equation*}
\mu = -\frac{\d\xi}{1+\xi^2}(X,x)-\frac{\xi}{1+\xi^2}(X,\d x)\qquad\text{and}\qquad g_N = \left(\frac{\d\xi}{1+\xi^2}\right)^2 + \frac{\xi^2}{1+\xi^2}g_{S^2}.
\end{equation*}
Changing coordinates to $\tilde{\xi}=\arctan\xi$ shows that these are the moment map and metric of $S^3$.}

We still need to solve \eqref{BPS1}. Using \eqref{BPS2} (still with $\gamma=0$ and $\alpha\neq0$), this is equivalent to $\star_{M} \d^A\phi = \big(1-3\frac{\alpha}{\beta}\big)\phi^{\ast A}\Sigma$ and takes the form
\begin{gather*}
\star_{M}\,\d\xi = \left(1-\frac{3\alpha}{\beta}\right)\frac{h_2^{\,2}f^2}{h_1}\,\omega_{S^2}, \\
\star_{M}\,\d x = \left(1-\frac{3\alpha}{\beta}\right)h_1 \d\xi\wedge x\,\d x.
\end{gather*}
Assuming that $\frac{3\alpha}{\beta}\neq1$, we can solve these equations by choosing
\begin{align*}
g_M &= \left(1-\frac{3\alpha}{\beta}\right)^2\bigl(h_1(\xi)^2\d\xi^2 + f(\xi)^2h_2(\xi)^2g_{S^2}\bigr).
\end{align*}

For later use, we point out that these spherical solutions can be reformulated in the language of spinors. The map $T_xS^2\to \operatorname{End}(\C^2)$ given by $X\mapsto x\,\d x(X)$ obeys the Clifford algebra relation $(x\,\d x(X))^2= (xX)^2 =-|X|^2\mathrm{Id}$, where we think of $x$ and $X$ as $\su(2)$ matrices satisfying $(x,x)=1$, $(x,X)=0$. Therefore, the trivial bundle $S=S^2\times \C^2$ is isomorphic to the spinor bundle and $x\,\d x(X)s=X\cdot s$ for any $s\in\C^2$. The endomorphism $\Phi$ corresponds to multiplication with the area 2-form $\omega_{S^2}$, and its eigenbundles are therefore the bundles $S^\pm$ of chiral spinors. The connection $\d-\frac12x\,\d x$ makes $\Phi$ parallel so restricts to an abelian connection on the bundles~$S^\pm$. Since its curvature is constant it must be the Levi-Civita connection of $S^2$. So when $f(\xi)=0$ the connection \eqref{spherical-ansatz} coincides with the spinorial connection \eqref{baer connection} for $C=S^2$.

\subsection{Symplectic solutions}
\label{sec:symplectic}
For our final family of solutions, let $C$ be a 2-manifold and let $L\to C$ be a hermitian line bundle. Let $a$ be a $\U(1)$ connection on $L$ with nowhere-vanishing curvature $\d a$. Then $\omega_C:=-2\,\d a$ determines a symplectic structure on $C$.

Let $J_\xi$ be a compatible family of almost complex structures on $C$ that depend on $\xi\in\II$. This means that, for each $\xi\in\II$ and $p\in C$, $J_{\xi,p}\colon T_pC\to T_p C$ satisfies $J_{\xi,p}^2=-1$ and that $g_{C,\xi}(u,v):=\omega_C(u,J_\xi v)$ is a family of riemannian metrics on $C$. Then $J_\xi$ determines a family of subbundles $\Lambda^{1,0}_\xi C\subset\C\otimes\Lambda^1C$. Let $w_\xi$ be a family of sections of $\Lambda^{1,0}_\xi C\otimes L^2$, normalised so that $2{\rm i}w_\xi\wedge\bar{w}_\xi=\omega_C$. Here $\bar{w}_\xi$ denotes the dual of $w$ with respect to the hermitian metric on $L^2$ and is a section of $\Lambda^{0,1}_\xi\otimes L^{-2}$.

We choose $M=\II\times C$, $\xi\colon \II\to \II$ to be the identify function, $E=L^\ast\oplus L$, and
\begin{equation*}
A=\begin{pmatrix}-{\rm i}a & -\bar{w}_\xi \\ \hphantom{-} w_\xi & \hphantom{-}{\rm i} a \end{pmatrix},\qquad \Phi=\begin{pmatrix}-{\rm i} & 0 \\ \hphantom{-}0 & {\rm i} \end{pmatrix}.
\end{equation*}
We assume that $\alpha=\gamma=\eta_2=0$ and that $\beta\neq0$. Then the BPS equation \eqref{BPS2} is equivalent to
\begin{equation*}
0 = \beta\,\phi^{\ast A}\mu^\sharp = \beta(\d a +{\rm i}w_\xi\wedge\bar{w}_\xi)\otimes\frac{\partial}{\partial\xi}.
\end{equation*}
This is automatically satisfied by our choice of $a$ and $w_\xi$ by construction. Since we now have that $\phi^{\ast A}\mu^\sharp=0$, equation \eqref{BPS1} is equivalent to $\star_M\phi^{\ast A}=\phi^{\ast A}\star_N$ and is solved by $g_M=\phi^{\ast A} g_N$. An explicit calculation shows that
\begin{equation*}
g_M = h_1(\xi)^2\d\xi^2 + 4h_2(\xi)^2 w_\xi\bar{w}_\xi = h_1(\xi)^2\d\xi^2 + h_2(\xi)^2 g_{C,\xi}.
\end{equation*}
This is a riemannian metric only if $w_\xi$ is non-vanishing. This means that $L^2$ is isomorphic to~$T^{1,0}_\xi C$ and that $2a$ is a connection on $T^{1,0}C$. If $C$ is compact, it further means that
\begin{equation*}
\chi(C) = c_1\big(T^{1,0}C\big) = \frac{{\rm i}}{2\pi}\int_C 2{\rm i}\,\d a = +\frac{1}{2\pi}\int_C \omega_C > 0,
\end{equation*}
so $C$ must be a 2-sphere.

\subsection{Classification of BPS solutions}

Having presented several solutions of \eqref{BPS1} and \eqref{BPS2}, we are now ready to state and prove a~theorem that classifies solutions.
\begin{Theorem}\label{thm:classification-S2}
Let $G=\SU(2)$ and let $N=\II\times \SU(2)/\U(1)$, with metric and moment map given in \eqref{metric-N-S2} and \eqref{moment-map-adjoint}. Suppose that $\eta_1(\xi)$ is non-vanishing and that $\alpha$, $\beta$, $\gamma$ are not all zero. Then any solution of the BPS equations \eqref{BPS1} and~\eqref{BPS2} is one of the solutions described above, i.e., either a Dirac monopole, a spinorial solution, a twisted spinorial solution, a spherical solution, or a symplectic solution.
\end{Theorem}
Before proving the theorem, we comment on the hypotheses. The hypothesis that $\eta_1$ is non-vanishing is motivated by the observation made below \eqref{eta1 constraint} that, if $N$ is compact, $\eta_1$ is non-zero. It is moreover reasonable to assume that $\alpha$, $\beta$, $\gamma$ are not all zero, as otherwise the BPS equation~\eqref{BPS2} is trivial.
\begin{proof} We begin by fixing a gauge where $\Phi$ is constant. Explicitly, we choose
\begin{equation}
\label{adjoint-Phi-A}
\Phi=\begin{pmatrix} -{\rm i} & 0 \\ \hphantom{-}0 & {\rm i} \end{pmatrix},\qquad
A = \begin{pmatrix} -{\rm i}a & -\bar{w} \\ \hphantom{-}w & \hphantom{-}{\rm i}a \end{pmatrix}.
\end{equation}
In this gauge, we obtain
\begin{equation}\label{adjoint-dPhi-F}
\d^A\Phi = \begin{pmatrix}\hphantom{-}0 & -2{\rm i}\bar{w} \\ -2{\rm i}w & \hphantom{-}0 \end{pmatrix}, \qquad
F = \begin{pmatrix} w \wedge \bar{w}-{\rm i}\d a & -\d^a\bar{w} \\ \d^a w& {\rm i}\d a - w\wedge\bar{w} \end{pmatrix},
\end{equation}
in which we have introduced $\d^a w = \d w + 2{\rm i}a\wedge w$ and $\d^a \bar{w} = \d \bar{w} - 2{\rm i}a\wedge \bar{w}$. The BPS equations~\eqref{BPS1} and \eqref{BPS2} take the form
\begin{align}
&\alpha \frac{h_2^{\,2}}{h_1}2{\rm i}w\wedge\bar{w} + \beta\frac{\eta_1}{h_1^{\,2}}(\d a+{\rm i}w\wedge\bar{w}) = 0, \label{adjoint-BPS1a}\\
&2\alpha h_1\d\xi\wedge w + \left(\beta\frac{\eta_2}{h_2^{\,2}}+2{\rm i}\gamma\right)\d^a w =0, \label{adjoint-BPS1b}\\
&\frac{h_2^{\,2}}{h_1}2{\rm i}w\wedge\bar{w} + 3\frac{\eta_1}{h_1^{\,2}}(\d a+{\rm i}w\wedge\bar{w}) = \star_{M}\d\xi,\nonumber \\ 
&2 h_1\d\xi\wedge w + 3\frac{\eta_2}{h_2^{\,2}}\,\d^a w =-2{\rm i}\star_{M}w.\label{adjoint-BPS2b}
\end{align}
We note that \eqref{adjoint-BPS2b} implies that{\samepage
\begin{equation}\label{adjoint-BPS3}
2\star_{M}|w|^2 = w\wedge\star_{M}\bar{w}+\bar{w}\wedge\star_{M}w = \tfrac{3{\rm i}\eta_2}{2h_2^{\,2}}\,\d(w \wedge\bar{w})+2{\rm i}h_1\,\d\xi\wedge w\wedge\bar{w}.
\end{equation}
We will solve the equations in two separate cases, according to the rank of $\d^A\phi$.}

First we consider the case where the rank of $\d^A\phi$ is less than 3. We will show that in this case all solutions have $w=0$. This means that the solutions are Dirac monopoles, as discussed in Section \ref{sec:Dirac monopoles}.

Since $\d^A\phi$ is not full-rank, $\phi^{\ast A}V_N=0$. This means that \begin{equation}\label{adjoint-notfullrank}
\d\xi\wedge w\wedge\bar{w}=0.
\end{equation}
If $\beta\eta_2\neq0$ or $\gamma\neq0$, then \eqref{adjoint-BPS1b} implies that
\begin{equation}\label{adjoint-dw-solution}
\d^a w = -\frac{2\alpha h_1h_2^{\,2}}{\beta \eta_2 + 2{\rm i}\gamma h_2^{\,2}}\,\d\xi\wedge w.
\end{equation}
Substituting this into \eqref{adjoint-BPS2b} gives
\begin{equation*}
\star_{M}w = {\rm i}h_1\left(1-\frac{3\alpha \eta_2}{\beta \eta_2 + 2{\rm i}\gamma h_2^{\,2}}\right)\d\xi\wedge w.
\end{equation*}
This, together with \eqref{adjoint-notfullrank}, implies that $\bar{w}\wedge\star_{M}w=0$, and hence that $|w|^2=0$.

On the other hand, if $\gamma=0$ and $\beta\eta_2=0$ then either $\beta=0$ or $\eta_2=0$. In the former case, we must have $\alpha\neq0$, and it follows from \eqref{adjoint-BPS1a} that $w\wedge\bar{w}=0$ and then from \eqref{adjoint-BPS3} that $|w|^2=0$. In the latter case, if $\eta_2=0$ it follows immediately from \eqref{adjoint-BPS3} and \eqref{adjoint-notfullrank} that $|w|^2=0$.

Now we consider the case where $\d^A\phi$ is full rank. Then $\d\xi$, $w$, $\bar{w}$ are a frame for the complexified cotangent bundle of $M$. Let $V$, $W$, $\bar{W}$ be the dual frame for $TM$. Then $\iota_V\d\xi=\iota_W w=\iota_{\bar{W}}\bar{w}=1$ and $\iota_W\d\xi=\iota_Vw=\iota_W\bar{w}=0$. Since $V$ is nowhere-vanishing it generates a flow on $M$. Let $C$ be the quotient of $M$ by this flow, which is homeomorphic to a level set of $\xi$ and hence is a 2-manifold. There is a natural projection $M\to C$. Combining this with the map $\xi\colon M\to\II$, we obtain a map
\begin{equation*}
M\lto \II\times C,
\end{equation*}
which is in fact a diffeomorphism. Under this diffeomorphism, $V$ is identified with the tangent vector $\partial/\partial\xi$ to $\II$, and we will write $V=\partial/\partial\xi$ from now on.

Consider the BPS equation \eqref{BPS2}. Each term is a section of $\Lambda^2M\otimes\phi^\ast TN$, or equivalently, a~map from $\Lambda^1N$ to $\Lambda^2M$. Since $\d^A\phi$ has full rank, we can invert $\phi^{\ast A}$ to obtain \smash{$\big(\phi^{\ast A}\big)^{-1}\colon \Lambda^1M\to\Lambda^1N$}. As explained in Section \ref{sec:nec-cond-BPS1}, composing each term in \eqref{BPS2} with \smash{$\big(\phi^{\ast A}\big)^{-1}$} results in the equation \eqref{metric-eqn}, each of whose terms is a map $\Lambda^1M\to\Lambda^2M$. As explained in Proposition \ref{prop:nec-for-BPS1}, the term involving the moment map must be trace-free, and we will see that this fact constrains the solutions.

First we write the constraint explicitly. From \eqref{moment-map-adjoint} and \eqref{adjoint-dPhi-F}, we obtain
\begin{equation*}
\phi^{\ast A}\mu^\sharp = \frac{\eta_1}{h_1^{\,2}}(\d a+{\rm i}w\wedge\bar{w})\otimes\frac{\partial}{\partial\xi}+\frac{\eta_2}{h_2^{\,2}}\begin{pmatrix}0 & -\d^a\bar{w} \\ \d^a w & 0 \end{pmatrix}.
\end{equation*}
In this equation, we are identifying the tangent space to $S^2$ at $\Phi$ with the space of matrices orthogonal to $\Phi$. From \eqref{adjoint-dPhi-F},
\begin{equation*}
\d^A\phi(W)=\begin{pmatrix}\hphantom{-} 0&0\\-2{\rm i}&0\end{pmatrix} \qquad\text{and}\qquad
\d^A\phi(\bar{W})=\begin{pmatrix}0&-2{\rm i}\\0&\hphantom{-}0\end{pmatrix},
\end{equation*}
while $\d^A\phi(\partial/\partial\xi)=\partial/\partial\xi$. Therefore,
\begin{equation*}
\big(\phi^{\ast A}\mu^\sharp\big)\circ \big(\phi^{\ast A}\big)^{-1} = \frac{\eta_1}{h_1^{\,2}}(\d a+{\rm i}w\wedge\bar{w})\otimes\frac{\partial}{\partial\xi}+\frac{{\rm i}\eta_2}{2h_2^{\,2}}\big(\d^a w\otimes W -\d^a\bar{w}\otimes\bar{W}\big).
\end{equation*}
The trace-free constraint from item (3) of Proposition \ref{prop:nec-for-BPS1} is
\begin{equation}
\label{adjoint-tracefree}
0 = \frac{\eta_1}{h_1^{\,2}}\iota_{\partial_\xi}\d a + \frac{{\rm i}\eta_2}{2h_2^{\,2}}\big(\iota_W\d^a w -\iota_{\ol{W}}\d^a\ol{w}\big),
\end{equation}
where we have made use of the fact that $\iota_{\partial\xi}w=\iota_{\partial\xi}\bar{w}=0$.

We now assume that at least one of $\beta\eta_2$ and $\gamma$ is non-zero (we will return to the case $\beta\eta_2=\gamma=0$ at the end of the proof). Then \eqref{adjoint-BPS1b} implies that $\d^aw$ is given by \eqref{adjoint-dw-solution}. Substituting this into the constraint \eqref{adjoint-tracefree} gives
\begin{equation*}
0 = \frac{\eta_1}{h_1^{\,2}}\iota_{\partial_\xi}\d a + \frac{4\alpha\gamma\eta_2 h_1h_2^{\,2}}{(\beta\eta_2)^2+(2\gamma h_2^{\,2})^2}\d\xi.
\end{equation*}
Contracting this with $\partial/\partial\xi$ tells us that $\alpha\gamma\eta_2=0$, and in turn that $\eta_1\iota_{\partial_\xi}\d a=0$. Then we are in one of the following three situations: either $\alpha=0$, $\gamma=0$, or $\eta_2=0$. We consider these in turn.

If $\alpha=0$, then $\d^a w=0$ from \eqref{adjoint-dw-solution}. Recall that $\iota_{\partial_\xi}w=0$. We may choose a gauge in which $\iota_{\partial_\xi}a=0$ also. Moreover,
\begin{equation*}
\begin{aligned}
\mathcal{L}_{\partial_\xi}a &= (\iota_{\partial_\xi}\d + \d\iota_{\partial_\xi})a=0 \qquad\text{and}\\
\mathcal{L}_{\partial_\xi}w &= (\iota_{\partial_\xi}\d + \d\iota_{\partial_\xi})w=\iota_{\partial_\xi}(\d^aw-2{\rm i}a\wedge w)=0.
\end{aligned}
\end{equation*}
Then $a$ and $w$ are both 1-forms on $C$ that have been pulled back to $M$, because their Lie derivatives and contractions with $\frac{\partial}{\partial \xi}$ vanish.

Consider now the metric $g_C=\Phi^{\ast A}g_{S^2}=4w\bar{w}$ on $C$. An orthonormal frame is given by $e^1=w+\bar{w}$, $e^2={\rm i}(\bar{w}-w)$. Then $\d^aw=0$ implies that
\begin{equation}
\label{torsion-free}
\d e^1 -2a\wedge e^2=0,\qquad \d e^2 +2a\wedge e^1=0.
\end{equation}
Equation \eqref{torsion-free} says that the metric-compatible connection $\nabla e^1:=2a\otimes e^2$, $\nabla e^2:=-2a\otimes e^1$ is torsion-free, so it must be the Levi-Civita connection of $g_C$. Moreover,
\begin{equation*}
TC\lto \operatorname{End}(E), \qquad X\mapsto \frac12\begin{pmatrix}0&-\iota_X\bar{w}\\ \iota_X w&0\end{pmatrix}
\end{equation*}
satisfies the Clifford algebra relation $X\cdot X=-g_C(X,X){\rm Id}$, so $(L^\ast\oplus L)$ is isomorphic to the spinor bundle of $C$. Then the ansatz \eqref{adjoint-Phi-A} is precisely the spinorial ansatz considered in Section~\ref{sec:spinorial}.

If $\gamma=0$ then $\d^aw$ is given by \eqref{adjoint-dw-solution}
\begin{equation}\label{adjoint-dw-solution-simplified}
\d^a w = -\frac{2\alpha h_1h_2^{\,2}}{\beta \eta_2}\d\xi\wedge w.
\end{equation}
Again, we choose a gauge in which $\iota_{\partial_\xi}a=0$ and find that $\mathcal{L}_{\partial_\xi}a=0$. On the other hand, \eqref{adjoint-dw-solution-simplified}~implies that
\begin{equation*}
\mathcal{L}_{\partial_\xi}w = (\iota_{\partial_\xi}\d + \d\iota_{\partial_\xi})w=-\frac{2\alpha h_1h_2^{\,2}}{\beta\eta_2}w,
\end{equation*}
where we used $\iota_{\partial_\xi}a=0$. This equation implies that $w=f(\xi)w_0$, where $w_0$ satisfies $\mathcal{L}_{\partial_\xi}w_0=\iota_{\partial_\xi}w_0=0$ and $f$ satisfies the differential equation \eqref{spherical-BPS2b}.

From \eqref{adjoint-BPS1a}, we obtain
\begin{equation*}
\d a = -{\rm i}\left(1+\frac{2\alpha h_1 h_2^{\,2}}{\beta\eta_1}\right)f^2\,w_0\wedge\bar{w}_0.
\end{equation*}
Since $\mathcal{L}_{\partial_\xi}\d a = \mathcal{L}_{\partial_\xi}w_0\wedge\bar{w}_0=0$, the real function $\big(1+\frac{2\alpha h_1 h_2^{\,2}}{\beta\eta_1}\big)f^2$ must be constant. Without loss of generality, we can choose $f,w_0$ so that this constant is 1. Then $\d^a w_0=0$, so, reasoning as above, $a$ must be the Levi-Civita connection for the metric $g_C=4w_0\bar{w}_0$. Since $\d a=-{\rm i}w_0\wedge\bar{w}_0$, this metric has scalar curvature $K=1$, i.e., it is the round metric on $S^2$. The map
\begin{equation*}
TM\lto \operatorname{End}(E), \qquad V\mapsto \frac12\begin{pmatrix}0&-\bar{w}_0(V)\\w_0(V)&0\end{pmatrix}
\end{equation*}
satisfies the Clifford algebra relation, so $E$ is the spinor bundle of $S^2$ and we are in the situation described at the end of Section \ref{sec:spherical}.

If $\eta_2=0$, then $\d^aw$ is given by \eqref{adjoint-dw-solution}
\begin{equation}\label{adjoint-dw-solution-simplified2}
\d^a w = {\rm i}\frac{\alpha h_1}{\gamma}\d\xi\wedge w.
\end{equation}
We choose the coordinate $\xi$ so that $h_1(\xi)=1$ and choose a gauge in which $\iota_{\partial_\xi}a=\alpha /2\gamma$. Then once again $\mathcal{L}_{\partial_\xi}a=0$, while
\begin{equation*}
\mathcal{L}_{\partial_\xi}w = (\d\iota_{\partial\xi}+\iota_{\partial\xi}\d)w = \iota_{\partial_\xi}(\d^aw - 2{\rm i}a\wedge w) = 0.
\end{equation*}
The connection $a_0=a-(\alpha/2\gamma)\d\xi$ satisfies $\iota_{\partial_\xi} a_0=0$, $\mathcal{L}_{\partial_\xi}a_0=0$ and $\d^{a_0}w=0$ so, reasoning as above, it must be the Levi-Civita connection for the metric $g_C=4w\bar{w}$. We are now in the situation described in Section \ref{sec:twisted}.

So, the solutions for which $\d^A\phi$ has full rank and at least one of $\beta\eta_2,\gamma$ is non-zero are those described in Sections \ref{sec:spinorial}, \ref{sec:twisted} and \ref{sec:spherical}. We now consider the case where $\beta\eta_2=\gamma=0$ and $\d^A\phi$ has full rank. In this case, \eqref{adjoint-BPS1b} implies that $\alpha=0$, as $\d\xi\wedge w\neq0$ by assumption. Then $\beta\neq0$, since we are assuming that $\alpha$, $\beta$, $\gamma$ are not all zero, and therefore $\eta_2=0$.

Since $\eta_1\neq0$, \eqref{adjoint-BPS1a} implies that $\d a = -{\rm i}w\wedge\bar{w}$. This in turn implies that $\iota_{\partial_\xi}\d a = 0$. Working in a gauge where $\iota_{\partial_\xi}a=0$, it follows that $\LL_{\partial_\xi}a=0$. Therefore, $a$ is the pullback of a connection on a line bundle over $C$. Its curvature is non-vanishing, because $\d^A\phi$ has full rank and therefore ${\rm i}\d\xi\wedge\d a=\d\xi\wedge w\wedge\bar{w}\neq 0$.

A short calculation shows that $2{\rm i}w\wedge\bar{w}=\Phi^{\ast A}\omega_{S^2}=:\omega_C$ and that $\Phi^{\ast A}g_{S^2}=4w\bar{w}=:g_C$. It follows that $w$ is a (1,0)-form with respect to the complex structure defined by the metric $g_C$ on $C$. Note that this metric and its complex structure could depend on $\xi$, because $\LL_{\partial_\xi}w$ may not vanish. However, the area form satisfies $\LL_{\partial_\xi}\omega_C = -2\LL_{\partial_\xi}\d a = 0$. We are now in precisely the situation considered in Section \ref{sec:symplectic}. Therefore, full-rank solutions in the case $\beta\eta_2=\gamma=0$ are the symplectic solutions of that section.
\end{proof}
\section[Principal orbit S\^{}3]{Principal orbit $\boldsymbol{S^3}$}\label{sec:princ-orbit-S3}

In this short final section, we consider the case where $G=\SU(2)$ acts on $N$ with principal orbit~$S^3$. In other words, we consider $N=\SU(2)$ acting on itself by left-multiplication. Without going into any of the details, the basic summary of this section is that the BPS equations~\eqref{BPS1} and~\eqref{BPS2} are not relevant at all in this setting, as the quantity \eqref{equivariant-top-degree} is not a well-defined topological charge. From a high-level perspective, this makes sense since the equivariant cohomology $H^3_{\SU(2)}\big(S^3,\R\big)=0$ is trivial. From a more explicit perspective, the issue that arises here is that there is no moment map $\mu$ which solves the constraint \eqref{moment-map-constraint}. So for this reason, we do not look for solutions of the equations \eqref{BPS1} and \eqref{BPS2} in this setting.

We now briefly discuss why there are no $\mu$ satisfying \eqref{moment-map-constraint}. The condition that $\SU(2)$ acts on $(\SU(2),g_{\SU(2)})$ by isometries forces $g_{\SU(2)}$ to be left-invariant. The volume form will thus be of the form
\begin{align*}
 V_{\SU(2)}=-\frac{K}{12}\tr(\theta_L\wedge\theta_L\wedge\theta_L)=-\frac{K}{12}\tr(\theta_R\wedge\theta_R\wedge\theta_R)
\end{align*}
for some $K>0$, where $\theta_{L,R}$ are the left- and right-invariant Maurer--Cartan forms discussed in Section \ref{sec:energy-gauged}. From \eqref{contraction-Maurer}, the vector field $\nu_L$ corresponding to the left-action of $\SU(2)$ on itself contracts with $V_{\SU(2)}$ to give
\begin{align*}
 \iota_{\nu_L(X)}V_{\SU(2)}=\tfrac{\begingroup\color{black}K\endgroup}{4}\tr(X\,\theta_R\wedge\theta_R)=\d\tfrac{\begingroup\color{black}K\endgroup}{4}\tr(X\,\theta_R),
\end{align*}
where in the last line we used the structure equation \eqref{structure-eqns-SU2}. It is straightforward to see that the only left-equivariant choice for $\mu$ satisfying $\d\mu(X)=\iota_{\nu_L(X)}V_{\SU(2)}$ is hence
\begin{align*}
 \mu(X)=\tfrac{\begingroup\color{black}K\endgroup}{4}\tr(X\,\theta_R).
\end{align*}
On the other hand, using \eqref{contraction-Maurer}, this satisfies
\begin{align*}
 \iota_{\nu_L(X)}\mu(X)=-\tfrac{\begingroup\color{black}K\endgroup}{4}\tr(X^2)=\tfrac{\begingroup\color{black}K\endgroup}{2}|X|^2\neq0,
\end{align*}
and so there are no $\mu$ satisfying both \eqref{moment-map-def} and~\eqref{moment-map-constraint}.
Although for this case we have shown that the topological charge cannot be realised as an~equivariant topological degree, it is worth remarking that an alternative topological description may be used by considering the difference between the ordinary topological degree \eqref{top-degree} and a~Chern--Simons functional; this is the convention taken by those studying \textit{electroweak skyrmions} (see, e.g., \cite{criadoKhozeSpannowsky2021emergence,dHokerFarhi1984decoupling}). Our focus has been on the BPS equations \eqref{BPS1} and~\eqref{BPS2} relevant for theories admitting topological charge given by the equivariant degree \eqref{equivariant-top-degree}, and so analysis of possible BPS equations relevant for electroweak skyrmions is reserved for future work.

\section{Concluding remarks}
We have described a geometric framework for gauged skyrmions on arbitrary space and target $3$-manifolds, generalising earlier work of Manton \cite{Manton1987geometry} on ordinary (ungauged) skyrmions to the setting of gauge theory. In particular, we have studied some BPS equations \eqref{BPS1} and~\eqref{BPS2} for gauged skyrmions, and classified all relevant solutions in the $\U(1)$- and $\SU(2)$-gauged case. Our results show that, unlike ordinary skyrmions which are only BPS when constant or isometries, a~rich variety of BPS solutions may be found, many of which are neither constant nor isometries.

We now comment on areas requiring further attention. Firstly, a limitation in our formulation is that it is explicitly $3$-dimensional, and so direct generalisation to other dimensions cannot be made which allow for BPS solutions. If one wanted to change the dimension of the target manifold (for example, to describe gauged hopfions), this would automatically change the degree of the forms $\Sigma$ and $\mu$, and then our BPS equations \eqref{BPS1} and~\eqref{BPS2} would not make sense. Therefore, any analogous study in these directions would require some modifications to what we have described.

A second limitation of our formulation is that we have restricted attention to a specific equivariant cohomology class, namely the class $[V_N+\mu]\in H^3_G(N,\R)$ associated with the volume form of $N$. Other classes could be interesting to study. One concrete example is electroweak skyrmions, associated with the left action of $\SU(2)$ on itself. As was noted above, in this case $H^3_{\SU(2)}(\SU(2),\R)=0$ but skyrmions on $\R^3$ can still be topologically non-trivial \cite{criadoKhozeSpannowsky2021emergence,dHokerFarhi1984decoupling}. Another example is $\SU(2)$ monopoles, for which $\SU(2)$ acts adjointly on $N=\su(2)$. In this case the topological degree is given by the integral of $\tr\big(F\wedge \d^A\phi\big)$, which is the pullback of the equivariant form $X\mapsto \tr(X\d x)$. This form represents a trivial class in the equivariant cohomology of $\su(2)\cong\R^3$, but it defines a nontrivial class in the relative equivariant cohomology of a ball $B^3\subset\R^3$ with boundary $S^2$.

Although we have focused attention on topological energy bounds, another important use of cohomology in gauge theory is in the construction of topological lagrangians. Just as the topological charge density $\tr(F\wedge F)$ of four-dimensional Yang--Mills theory can be used to define a three-dimensional Chern--Simons lagrangian, so too can other $n$-dimensional topological charge densities be used to define $(n-1)$-dimensional Lagrangians. The paper \cite{tchrakian2022SkyChernSimons} seeks to classify lagrangians for gauged sigma-models that arise in this way. Equivariant cohomology, as reviewed in this article, could be used to systematically classify such lagrangians.

Finally, this article has focused on generalising to gauge theory the ordinary Skyrme energy~\eqref{Skyrme-energy} consisting of only two terms -- the Dirichlet term $|\d\phi|^2$, and Skyrme term $|\phi^\ast\Sigma|^2$ -- which are quadratic and quartic in derivatives, and are constantly coupled. Over the years, all sorts of different Skyrme energies have been proposed which include additional terms which do not necessarily adhere to this constraint, for example by adding higher or lower order terms, or including field-dependent couplings (see, e.g., \cite{AdamOlesWereszczynski2020dielectric,AdamSanchez-GuillenWereszczynski2010sexticmodel,GudnasonNitta2015baryonic,harland2014topological}). Of particular interest in physically-realistic models is the inclusion of a potential term which gives the pion mass \cite{adkinsNappi1984pionmass}, and recently some work has been done in understanding gauged skyrmions with massive pions \cite{LivramentoRaduShnir2023solitons}. It would be interesting to consider all of these generalised Skyrme models within a geometric framework in a similar vein to what we have described here.
\subsection*{Acknowledgements}
The authors would like to thank Nuno Rom\~ao for useful discussions during early stages of this work.

\pdfbookmark[1]{References}{ref}
\LastPageEnding

\end{document}